\documentclass{article} 

\usepackage{fullpage} 
\usepackage{amssymb,amsmath}
\usepackage{amsthm} 
\usepackage{graphicx}

\usepackage{hyperref} 

\usepackage{algorithm} 
\usepackage{algorithmic} 
\usepackage{enumerate}
\usepackage{epstopdf}

\newtheorem{theorem}{Theorem} 
\newtheorem{lemma}{Lemma} 
\newtheorem{proposition}{Proposition} 
\theoremstyle{definition}
\newtheorem{remark}{Remark}
\newtheorem{example}{Example}

\newcommand{\email}[1]{{\texttt #1}} 

\newcommand{\TheTitle}{A Coordinate Descent Primal-Dual Algorithm with Large Step Size and Possibly Non-Separable Functions}

\title{{\TheTitle}
\thanks{
A summary of the results of this paper 
has been published in the proceedings of the 
2016 Conference on Decision and Control~\cite{bianchi2016using}.
This work has been supported by the Orange/Telecom ParisTech think tank Phi-TAB.
}}

\author{
  Olivier Fercoq\thanks{LTCI, CNRS, T\'el\'ecom ParisTech, Universit\'e Paris-Saclay, 75013, Paris, France, (\email{olivier.fercoq@telecom-paristech.fr}). }
  \and
  Pascal Bianchi\thanks{LTCI, CNRS, T\'el\'ecom ParisTech, Universit\'e Paris-Saclay, 75013, Paris, France, (\email{pascal.bianchi@telecom-paristech.fr})}
}

\newtheorem{assum}{Assumption}[section]

\providecommand{\norm}[1]{\left\lVert#1\right\rVert}

\newcommand{\prox}{\mathrm{prox}}

\newcommand{\bs}{\boldsymbol}

\newcommand{\cX}{{\mathcal X}}
\newcommand{\cY}{{\mathcal Y}}

\newcommand{\cA}{{\mathcal A}}

\newcommand{\cF}{{\mathcal F}}

\newcommand{\cC}{{\mathcal C}}

\newcommand{\cS}{{\mathcal S}}

\newcommand{\bP}{{\mathbb P}}
\newcommand{\bR}{{\mathbb R}}

\newcommand{\bE}{{\mathbb E}}

\newcommand{\bN}{{\mathbb N}}

\newcommand{\ps}[1]{\langle #1 \rangle}

\newcommand{\dom}{\mathrm{dom}}

\newcommand{\revision}[1]{#1} 
\newcommand{\revisiontrois}[1]{#1} 

\begin{document}

\maketitle

\begin{abstract}
This paper introduces a randomized coordinate descent version of the V\~{u}-Condat algorithm. By coordinate descent, we mean that only a subset of the coordinates of the primal and dual iterates is updated at each iteration, the other coordinates being maintained to their past value. Our method allows us to solve optimization problems with a combination of differentiable functions, constraints as well as non-separable and non-differentiable regularizers. 

We show that the sequences generated by our algorithm almost surely converge to a saddle point of the problem at stake, for a wider range of parameter values than previous methods. In particular, the condition on the step-sizes depends on the coordinate-wise Lipschitz constant of the differentiable function's gradient, which is a major feature allowing classical coordinate descent to perform so well when it is applicable. 
\revisiontrois{We then prove a sublinear rate of convergence in general and 
a linear rate of convergence if the objective enjoys strong convexity 
properties.}

We illustrate the performances of the algorithm on a total-variation regularized least squares regression problem and on large scale support vector machine problems.
\end{abstract}


\section{Introduction}

\subsection{Motivation}


We consider the optimization problem
\begin{equation}
\inf_{x\in \cX} f(x)+g(x)+h(Mx)
\label{eq:pb}
\end{equation}
where $\cX$ is a Euclidean space, $M:\cX\to \cY$ is a linear operator onto a second Euclidean space $\cY$;
functions $f:\cX\to \bR$, $g:\cX\to (-\infty,+\infty]$ and $h:\cY\to ]-\infty,+\infty]$ are assumed proper, closed and convex;
the function $f$ is moreover assumed differentiable. We assume that $\cX$ and $\cY$ are product spaces of the form
$\cX=\cX_1\times\cdots\times \cX_n$ and $\cY=\cY_1\times\cdots\times \cY_p$ for some integers $n,p$.
For any $x\in \cX$, we use the notation $x=(x^{(1)},\dots,x^{(n)})$ to represent the (block of) coordinates of $x$ (similarly for 
$y=(y^{(1)},\dots,y^{(p)})$ in $\cY$). Problem~(\ref{eq:pb}) has numerous applications \emph{e.g.} in machine learning~\cite{cevher2014convex},
image processing~\cite{chambolle2010introduction} or distributed optimization~\cite{boyd2011distributed}.

Under the standard qualification condition $0\in \mathrm{ri}(M\mathrm{dom}g-\mathrm{dom}h)$ (where $\mathrm{dom}$ and 
$\mathrm{ri}$ stand for domain and relative interior, respectively), a point $x\in \cX$ is a minimizer of~(\ref{eq:pb}) if and only if 
there exists $y\in \cY$ such that $(x,y)$ is a saddle point of the Lagrangian function
$$
L(x,y) = f(x)+g(x)+\ps{y,Mx}-h^\star(y)
$$
where $\ps{\,.\,,.\,}$ is the inner product and $h^\star:y\mapsto \sup_{z\in \cY}\ps{y,z}-h(z)$ is the Fenchel-Legendre transform of~$h$.
There is a rich literature on \emph{primal-dual} algorithms searching for a saddle point of $L$ (see \cite{tran2014primal} and references therein).
In the special case where $f=0$, the alternating direction method of multipliers (ADMM)
proposed by Glowinsky and Marroco~\cite{glowinski1975approximation}, Gabay and Mercier \cite{gabay1976dual}
and the algorithm of Chambolle and Pock \cite{chambolle2011first} are amongst the most celebrated ones.
Based on an elegant idea also used in~\cite{he2012convergence}, 
V\~u~\cite{vu2013splitting} and Condat~\cite{condat2013primaldual} separately proposed a primal-dual algorithm allowing
as well to handle $\nabla f$ explicitly, and requiring one evaluation of the gradient of $f$ at each iteration.
Hence, the $\nabla f$ is handled explicitly in the sense that the algorithm does \emph{not} involve, for instance, the call of a proximity operator 
associated with $f$.
A convergence rate analysis is provided in \cite{chambolle2014ergodic} (see also~\cite{tran2014primal}).
A related splitting method has been recently introduced by \cite{davis2015three}.

This paper introduces a \emph{coordinate descent} (CD) version of the
V\~u-Condat algorithm.  By coordinate descent, we mean that only a
subset of the coordinates of the primal and dual iterates is updated
at each iteration, the other coordinates being maintained to their
past value. Coordinate descent was historically used in the context of
coordinate-wise minimization of a unique function in a Gauss-Seidel
sense
\cite{warga1963minimizing,bertsekas1989parallel,Tseng01convergence}.
Tseng \emph{et al.} \cite{Tseng:CCMCDM:Smooth,Tseng:CGDM:Nonsmooth,Tseng:2010:CGD:1861477.1861509}
 and Nesterov \cite{Nesterov:2010RCDM}
developped CD versions of the gradient descent. 
In \cite{Nesterov:2010RCDM} as well as in this paper, the updated
coordinates are randomly chosen at each iteration.  The algorithm of
\cite{Nesterov:2010RCDM} has at least two interesting features. Not
only it is often easier to evaluate a single coordinate of the
gradient vector rather than the whole vector, but the conditions under
which the CD version of the algorithm is provably convergent are
generally weaker than in the case of standard gradient descent.  The key
point is that the \emph{step size} used in the algorithm when updating
a given coordinate $i$ can be chosen to be inversely proportional to
the \emph{coordinate-wise} Lipschitz constant of $\nabla f$ along its
$i$th coordinate, rather than the global Lipschitz constant of $\nabla
f$ (as would be the case in a standard gradient descent). Hence, the
introduction of coordinate descent allows to use \emph{\revision{longer} step
  sizes} which potentially results in a more attractive performance.
The random CD gradient descent of~\cite{Nesterov:2010RCDM} was later
generalized by Richt{\'a}rik and Tak{\'a}{\v c}~\cite{RT:UCDC} to the
minimization of a sum of two convex functions $f+g$ (that is, $h=0$ in
problem~(\ref{eq:pb})).  The algorithm of~\cite{RT:UCDC} is
analyzed under the additional assumption that function $g$ is
\emph{separable} in the sense that for each $x\in \cX$, $g(x) =
\sum_{i=1}^ng_i(x^{(i)})$ for some functions $g_i:\cX_i\to
]-\infty,+\infty]$.  Accelerated and parallel versions of the
algorithm have been later developed by
\cite{RT:TTD2011,RT:PCDM,FR:2013approx,lin2014accelerated}, always assuming
the separability of~$g$.

In the literature, several papers seek to apply the principle of coordinate descent to
primal-dual algorithms.
In the case where $f=0$, $h$ is separable and smooth and $g$ is
strongly convex, Zhang and Xiao \cite{zhang2014stochastic} introduce
 a stochastic CD primal-dual algorithm and analyze its convergence rate
(see also~\cite{suzuki2014stochastic} for related works).
In 2013, Iutzeler et al. \cite{iutzeler2013asynchronous} proved
that random coordinate descent can be successfully applied to fixed point iterations
of firmly non-expansive (FNE) operators. \revision{According} to~\cite{gabay1983chapter},
the ADMM can be written as a fixed point algorithm of a
FNE operator, which led the authors of~\cite{iutzeler2013asynchronous} to propose a coordinate descent version of ADMM with
application to distributed optimization. The key idea behind the
convergence proof of \cite{iutzeler2013asynchronous} is to establish the so-called stochastic
Fej\'er monotonicity of the sequence of iterates as noted by~\cite{combettes2014stochastic}. 
In a more general setting than \cite{iutzeler2013asynchronous}, 
Combettes \emph{et al.} in~\cite{combettes2014stochastic} and Bianchi \emph{et al.} 
\cite{bianchi2014stochastic} extend the proof to the so-called 
$\alpha$-averaged operators, which include FNE operators as a
special case. This generalization allows to apply the coordinate descent principle to
a broader class of primal-dual algorithms which is no longer restricted to 
the ADMM or the Douglas Rachford algorithm. 
For instance, Forward-Backward splitting is considered in~\cite{combettes2014stochastic}
and \revision{particular cases of the V\~u-Condat algorithm are considered in~\cite{bianchi2014stochastic, pesquet2014randomprimaldual}.
Nevertheless, the above approach has two major limitations.

First, in order to derive a converging coordinate descent version of a given deterministic algorithm,
the latter must write as a fixed point algorithm over some product Hilbert space of the form $H=H_1\times \cdots H_q$
where the inner product in $H$ is the sum of the inner products in the $H_i$'s.
Unfortunately, this condition  does {\em not} hold in general 
for the V\~{u}-Condat method, because the inner product over $H$ involves the coupling linear operator $M$.
A workaround was proposed in~\cite{bianchi2014stochastic}
but for a particular example only.

Second and even more importantly, 
the approach of~\cite{iutzeler2013asynchronous,combettes2014stochastic,bianchi2014stochastic,pesquet2014randomprimaldual}
needs ``small'' step sizes. More precisely, the convergence conditions are identical to the ones of the brute method, the one without coordinate descent.
These conditions  involve the global Lipschitz constant of the gradient $\nabla f$ instead than its coordinate-wise Lipschitz constants.
In practice, it means that the application of coordinate descent to primal-dual algorithm as 
suggested by~\cite{combettes2014stochastic} and \cite{bianchi2014stochastic}
is restricted to the use of potentially small step sizes.
One of the major benefits of coordinate descent is lost. 

\revisiontrois{Some recent works also focused on designing primal-dual coordinate 
descent methods with a guaranteed convergence rate. In~\cite{gao2016randomized} and~\cite{chambolle2017stochastic},
a $O(1/k)$ rate is obtained for the ergodic mean of the sequences.
The rates are given in terms of feasibility and optimality or Bregman distance.
Those two papers require all the dual variables to be updated at each iteration, which may
not be efficient if there are more than a few dual variables.
In the present paper, we will have much more flexibility in the variables we
choose to update at each iteration, while retaining a provable convergence rate.
}

\subsection{Contribution}

\begin{itemize}
\item Our main contribution is to provide a CD primal-dual algorithm with a broad
range of admissible step sizes. Our numerical experiments show that remarkable performance gains
can be obtained when using larger step sizes.
\item \revisiontrois{We identify two setups for which the structure of the problem is favorable to coordinate descent algorithms.}
\item \revisiontrois{We prove a sublinear rate of convergence in general and 
	a linear rate of convergence if the objective enjoys strong convexity 
	properties.}
\end{itemize}

}


\subsection{Organization of the paper}

The algorithm is introduced in
Section~\ref{sec:algo}. At each iteration $k$, an index $i$ is
randomly chosen w.r.t. the uniform distribution in $\{1,\dots,n\}$
where $n$ is, as we recall, the number of primal coordinates.  The
coordinate $x_k^{(i)}$ of the current primal iterate $x_k$ is updated,
as well as a set of associated dual iterates.  Under some assumptions
involving the coordinate-wise Lipschitz constants of $\nabla f$, the
primal-dual iterates converges to a saddle point of the Lagrangian. As a remarkable feature, our CD algorithm
makes no assumption of separability of the functions $f$, $g$ or $h$. In
the special case where $h=0$ and $g$ is separable, the algorithm reduces to the CD proximal
gradient algorithm of \cite{RT:UCDC}.

The convergence proof is provided in Section~\ref{sec:proof}. It is
worth noting that, under the stated assumption on the step-size, the
stochastic \revision{Fej\'er} monotonicity of the sequence of iterates, which is
the key idea in \cite{iutzeler2013asynchronous,combettes2014stochastic,bianchi2014stochastic},
 does not hold (a
counter-example is provided). Our proof relies on the introduction of
an adequate Lyapunov function. 
In Section~\ref{sec:rate}, we prove a sublinear rate of convergence in general and 
	a linear rate of convergence if the objective enjoys strong convexity 
	properties.
 In Section~\ref{sec:simu}, the
proposed algorithm is instantiated to the case of total-variation  regularization
and support vector machines.  Numerical results performed on real MRI
and text data establish the attractive behavior of the proposed algorithm and
emphasize the importance of using primal-dual CD with large step
sizes.

\section{Coordinate Descent Primal-Dual Algorithm}
\label{sec:algo}

\subsection{Notation}

 We note $M = (M_{j,i}:j \in\{1,\dots, p\},\,i\in\{1,\dots, n\})$ where $M_{j,i}:\cX_i\to\cY_j$ are the block components of $M$.
For each $j \in \{1,\dots,p\}$, we introduce the set
$$
I(j) := \Big\{i \in \{1,\dots,n\} \,:\,M_{j,i}\neq 0\Big\}\,.
$$
Otherwise stated, the $j$th component of vector $Mx$ only depends on $x$ through the coordinates $x^{(i)}$ such that $i\in I(j)$.
We denote by
$$
m_j := \mathrm{card}(I(j))
$$
the number of such coordinates. Without loss of generality, we assume that $m_j\neq 0$ for all $j$.
\revisiontrois{We also denote
\[
\pi_j := \frac{1}{\mathrm{card}(I(j))} \,.
\]
}

For all $i \in \{1,\dots,n\}$, we define
\[
J(i) := \Big\{ j \in \{1,\dots,p\}\,:\, M_{j,i}\neq 0\Big\}\,.
\]
Note that for every pair $(i,j)$, the statements $i\in I(j)$ and $j\in J(i)$ are equivalent.

If $\ell$ is an integer, 
$\gamma= (\gamma_1,\dots,\gamma_\ell)$ is a collection of positive
real numbers and $\cA = \cA_1\times \dots\times \cA_\ell$ is a product of Euclidean spaces, 
we introduce the weighted norm $\|\,.\,\|_\gamma$ on $\cA$ given by
$\|u\|^2_\gamma = \sum_{i=1}^\ell \gamma_i \|u^{(i)}\|^2_{\cA_i}$ for every
$u=(u^{(1)},\dots,u^{(\ell)})$ where $\|\,.\,\|_{\cA_i}$ stand for the norm on $\cA_i$.
If $F:\cA\to ]-\infty,+\infty]$ denotes a convex proper lower-semicontinuous
function, we introduce the proximity operator $\prox_{\gamma,F}:\cA\to\cA$ defined 
for any $u\in \cA$ by
$$
\prox_{\gamma,F}(u) := \arg\min_{w\in \cA} \big[ F(w)+\frac 12\|w-u\|^2_{\gamma^{-1}} \big]
$$
where we use the notation $\gamma^{-1} = (\gamma_1^{-1},\dots,\gamma_\ell^{-1})$.
We denote by $\prox_{\gamma,F}^{(i)}:\cA\to\cA_i$ the $i$th coordinate mapping of $\prox_{\gamma,F}$ that is,
$\prox_{\gamma,F}(u)=(\prox_{\gamma,F}^{(1)}(u),\dots,\prox_{\gamma,F}^{(\ell)}(u))$ for any $u\in \cA$.
The notation $D_{\cA}(\gamma)$ (or simply $D(\gamma)$ when no ambiguity occurs) stands
for the diagonal operator on $\cA\to\cA$ given by $D_{\cA}(\gamma)(u) = (\gamma_1 u^{(1)},\dots,\gamma_\ell u^{(\ell)})$ for every
$u=(u^{(1)},\dots,u^{(\ell)})$.

Finally, the adjoint of a linear operator $B$ is denoted $B^\star$. 
The spectral radius of a square matrix $A$ is denoted by $\rho(A)$.
The number of nonzero elements of a matrix $A$ is denoted by $\mathrm{nnz}(A)$.

\subsection{Main algorithm}

Consider Problem~(\ref{eq:pb}).
Let $\sigma = (\sigma_1,\dots,\sigma_p)$ and $\tau=(\tau_1,\dots,\tau_n)$ be two tuples of positive real numbers.
Consider an independent and identically distributed sequence $(i_k\,:\, k\in \bN^*)$ with uniform distribution on 
$\{1,\dots,n\}$\footnote{The results of this paper easily extend to the selection of several primal coordinates
	at each iteration with a uniform samplings of the coordinates, 
	using the techniques introduced in \cite{RT:PCDM}.}.
The proposed primal-dual CD algorithm consists in updating two sequences $x_k\in \cX$, $y_k\in \cY$.
It is provided in Algorithm~\ref{algo:notduplicated} below.

\revisiontrois{
\begin{algorithm}[ht]
	\caption{Coordinate-descent primal-dual algorithm}
	\label{algo:notduplicated}
\revisiontrois{
	\noindent {\bf Initialization}: Choose $x_0\in \cX$, $y_0\in \cY$. \\ 
	\noindent {\bf Iteration $k$}: Define:
	\begin{align*}
	\overline y_{k+1} &= \prox_{\sigma,h^\star}\big(y_k+D(\sigma) Mx_k\big) \\
	\overline x_{k+1} &= \prox_{\tau,g}\Big(x_k-D(\tau)\left(\nabla f(x_k)+2M^\star \overline {y}_{k+1}- M^\star y_k\right)\Big)\,.
	\end{align*}
	For $i=i_{k+1}$ and for each $j \in J(i_{k+1})$, update:
	\begin{align*}
	&{x}_{k+1}^{(i)} = \overline {x}_{k+1}^{(i)} \\
	&y_{k+1}^{(j)} = y_{k}^{(j)} + \pi_{j} (\overline {y}_{k+1}^{(j)} - y_{k}^{(j)})\,.
	\end{align*}
	Otherwise, set ${x}_{k+1}^{(i')}=x_k^{(i')}$, and 
	$y_{k+1}^{(j')}=y_{k}^{(j')}$.
	}
\end{algorithm}
}

For every $i\in \{1,\dots,n\}$, we denote by $U_i:\cX_i\to\cX$ the linear operator such that all coordinates of $U_i(u)$ are zero except the $i$th coordinate
which coincides with $u$:
$U_i(u)=(0,\cdots,0, u, 0, \cdots,0)$.   
Our convergence result holds under the following assumptions.
\begin{assum}
	\label{hyp:main}
	\begin{enumerate}[a)]
		\item \label{hyp:main-convex} The functions $f$, $g$, $h$ are closed proper and convex.
		\item \label{hyp:main-diff} The function $f$ is differentiable on $\cX$.
		\item \label{hyp:main-smooth} For every $i \in \{1,\dots,n\}$, there exists $\beta_i\geq 0$ such that for any $x\in \cX$, any $u\in \cX_i$,
		\[
		f(x+U_i u)\leq f(x) + \ps{\nabla f(x), U_i u}+\frac{\beta_i}2\|u\|^2_{\cX_i}\,.
		\]
		\item \label{hyp:main-unif} The random sequence $(i_k)_{k\in \bN^*}$ is independent,  uniformly distributed on  $\{1,\dots,n\}$.
		\item \label{hyp:main-norm} The step sizes $\tau = (\tau_1, \ldots, \tau_n)$ and $\sigma = (\sigma_1, \ldots, \sigma_p)$ satisfy for all $i \in \{1,\dots,n\}$,
		\[
		\tau_i < \frac 1{\beta_i+\rho\left(\sum_{j\in J(i)}\revisiontrois{(2-\pi_{j})}m_j\sigma_j M_{j,i}^\star M_{j,i}\right)}\,.
		\]
	\end{enumerate}
\end{assum}
We denote by $\cS$ the set of saddle points of the Lagrangian function $L$. 
Otherwise stated, a couple $(x_*,y_*)\in \cX\times\cY$ lies in $\cS$ if and only if it satisfies the following inclusions
\begin{align}
0&\in \nabla f(x_*)+\partial g(x_*)+M^\star y_* \\
0&\in -Mx_*+\partial h^\star(y_*)\,.
\label{eq:pd}
\end{align}
We shall also refer to elements of $\cS$ as primal-dual solutions.

\revisiontrois{
\begin{theorem}
	\label{the:notduplicated}
	Let Assumption~\ref{hyp:main} hold true and suppose that $\cS\neq\emptyset$. Let $(x_k, y_k)$ 
	be a sequence generated by Algorithm~\ref{algo:notduplicated}. Almost surely,
	there exists $(x_*,y_*)\in \cS$ such that
	\begin{align*}
	&\lim_{k\to\infty} x_k = x_* \\
	&\lim_{k\to\infty} y_k = y_* \,.
	\end{align*}
\end{theorem}

}

\subsection{Efficient implementation using problem structure}

In Algorithm~\ref{algo:notduplicated}, it is worth noting that quantities $(\overline x_{k+1},\overline y_{k+1})$
do not need to be explicitly calculated. At iteration $k$, only the coordinates
\[
\overline x_{k+1}^{(i_{k+1})}\text{  and  }\,\overline y_{k+1}^{(j)}, \ \ \forall j\in J(i_{k+1})
\]
are needed to perform the update. 
From a computational point of view, it is often the case that the evaluation of the above coordinates
is less demanding than the computation of the whole vectors $\overline x_{k+1}$, $\overline y_{k+1}$. Two situations have been reported in the literature:
\revisiontrois{
\begin{itemize}
\item If $g$ is separable, one only needs to compute the quantities $\nabla_{i_{k+1}} f(x_k)$, $(2M^\star \bar y_{k+1} - M^\star y_k)^{(i_{k+1})}$ and $\prox_{\tau_{i_{k+1}}, g_{i_{k+1}}}$ to perform the $k$th iteration. 
A~classical example of such smart residual update~\cite{Nesterov-Subgrad-Huge} can be found in the proximal coordinate descent gradient algorithm (case $g$ separable and $h = 0$) \cite{richtarik2014iteration}. 
More generally, if $g$ (resp. $h^\star$) is 
block-separable, we can use this structure in the algorithm, even if
this block structure does not match $\cX_1 \times \ldots \times \cX_n$ (resp. $\cY_1 \times \ldots \times \cY_p$).

We used this idea in Section~\ref{sec:exp_tv} to deal efficiently with
the proximal operator of the $\ell_{2,1}$ norm.
	
\item If $g$ is the indicator of the consensus constraint $\{x_1 = \dots = x_n\}$, $f$ is separable and $h=0$, we recover MISO~\cite{mairal2015incremental}. In that case, we can store $\nabla f(x_k)$ and update its average. Thanks to the separability of $f$, only
one coordinate of $\nabla f(x_k)$ needs to be updated at each iteration.
	
We used similar ideas in Section~\ref{sec:svm} to deal efficiently with
the projection onto the subspace orthogonal to a vector.
\end{itemize}

To illustrate the importance of these implementation 
tricks, we give in the following table a comparison of the number of operations to compute the updates
of the standard V\~u-Condat method against the proposed algorithm.

}

\begin{table}[h]
  \centering
  \begin{tabular}{|c|c|c|}
\hline
Problem / Dimension of data & V\~u-Condat & Our algorithm \\
\hline
Total Variation + $\ell_1$ regularization  & $O(mn + 6n)$ &  $O(m +  12)$ \\
$A \in \mathbb R^{m \times n}$:\! dense; $M \in \mathbb R^{3 n \times n}$:\! $\mathrm{nnz}(M) = 6n$ & & \\
\hline
Support Vector Machines & $O(\mathrm{nnz}(A) + n)$& $O(\mathrm{nnz}(Ae_i) + 1)$ \\
$A \in \mathbb R^{m \times n}$: sparse && \\
\hline
  \end{tabular}
\smallskip
  \caption{\small Number of operations per iteration for the proposed algorithm and for the standard V\~u-Condat algorithm
- The use cases are the ones described in the numerical section. 
The numbers 6 and 12 highlight the (mild) overhead of duplication in the Total Variation + $\ell_1$ regularized least squares problem.}
\end{table}

\subsection{Primal dual coordinate descent with duplicated dual variables}

\revisiontrois{
In this section, we present a generalization of Algorithm~\ref{algo:notduplicated} that allows for more flexibility in
the update rule for dual variables. It will also be a convenient formulation for the analysis.
}

Recall that $\cY=\cY_1\times\cdots\times \cY_p$.
For every $j\in \{1,\dots,p\}$, we use the notation $\bs \cY_j:=\cY_j^{I(j)}$, 
\revisiontrois{which means that $\bs \cY_j$ consists of $|I(j)|$ copies of $\cY_j$ indexed by $I(j)$.}
An arbitrary element $\bs u$ in $\bs \cY_j$ will be represented by $\bs u=(\bs u{(i)}\,:\,i\in I(j))$. 
We define $\bs \cY:=\bs \cY_1\times \cdots\times \bs\cY_p$. An arbitrary element $\bs y$ in $\bs \cY$ will be represented
as $\bs y=(\bs y^{(1)},\dots,\bs y^{(p)})$ and we shall call such an element a duplicated dual variable. This notation is recalled in Table~\ref{tab:indices} below.
\begin{table}[ht]
\centering
\caption{Standing notation.}
\begin{tabular}[h]{|c|c|c|}
\hline
Space & Element & \revisiontrois{Dimension}  \\
&         & \revisiontrois{(if blocks of size 1)} \\
\hline
$\cX = \cX_1\times\cdots\times \cX_n$ & $x= (x^{(i)}:i \in \{1,\dots,n\})$ & \revisiontrois{$n$} \\
\hline
$\cY = \cY_1\times\cdots\times \cY_p$ & $y= (y^{(j)}:j \in \{1,\dots,p\})$ & \revisiontrois{$p$} \\
\hline
$\bs\cY_j = \cY_j^{I(j)}$ & $\bs u= (\bs u(i):i\in I(j))$ & \revisiontrois{$|I(j)|$} \\
\hline
$\bs\cY = \bs\cY_1\times\cdots\times \bs\cY_p$ & $\bs y= (\bs y^{(j)}:j \in \{1,\dots,p\})$ & \revisiontrois{$\mathrm{nnz}(M)$} \\ & where $\bs y^{(j)}= (\bs y^{(j)}(i):i\in I(j))$ $\forall j$ & \\
\hline
\end{tabular}
\label{tab:indices}
\end{table}

\revisiontrois{In our algorithm, we will stack a collection 
	of primal variables $(x_k^{(i)}:i\!\in\! \{1,\dots,n\})$ at iteration $k$,
	and a set of (duplicated) dual variables 
	$(\bs y_k^{(j)}(i):i\in \{1,\dots,n\},j\in J(i))$.
	In a coordinate descent spirit, we however update only a subset of these variables
	at every iteration $k$. First, we choose uniformly at random a
	block of primal coordinates $i_{k+1}$: eventually, only the primal variable $x_k^{(i_{k+1})}$ will be updated.
	As far as the dual variables are concerned, a natural choice is to update the dual variables
	$(\bs y_k^{(j)}(i_{k+1}):j\in J(i_{k+1}))$ associated to the primal variable $x_k^{(i_{k+1})}$. 
	This case will be investigated in Section~\ref{sec:Jminimal}. 
	For reasons that will be made clear later on, it may be interesting in some situations
	to update a larger set of duplicated dual variables at iteration $k$, namely
	$(\bs y_k^{(j)}(l):(l,j)\in {\mathcal J}(i_{k+1}))$ where for every $i\in \{1,\dots,n\}$,
	$\mathcal J(i)$ is a subset of $\{1,\dots,n\}\times\{1,\dots,p\}$ chosen in such a way that
	\begin{equation}
	\{i\}\times J(i)\,\subset\,\mathcal J(i)\,\subset\, \left\{(l,j):j\in J(l)\right\}\,.
	\label{eq:constraintJ}
	\end{equation}
	We shall also
	define the probability that $j \in J(i_{k+1})$ knowing that 
	$(l,j) 	\in \mathcal J(i_{k+1})$ as
	\begin{equation}
	\label{eq:def_pi_j}
	\bs \pi_{j}(i) = \frac{1}{\mathrm{card}(\{l \,:\, (i,j) \in \mathcal J(l) \})} \,.
	\end{equation}
	Note that $0 < \bs \pi_{j}(i) \leq 1$. In the special case where $\mathcal J(i)=\{i\}\times J(i)$, note also that $\bs \pi_{j}(i)=1$ for every $j\in J(i)$.
}

As for Algorithm~\ref{algo:notduplicated}, we consider an independent and identically distributed sequence $(i_k\,:\, k\in \bN^*)$ with uniform distribution on 
$\{1,\dots,n\}$.
The algorithm consists in updating four sequences $x_k\in \cX$, $w_k\in \cX$, $z_k\in \cY$ and $\bs y_k\in \bs\cY$.
It is provided in Algorithm~\ref{algo:main} below.

\begin{algorithm}[ht]
\caption{Coordinate-descent primal-dual algorithm \revisiontrois{with duplicated variables}}
\label{algo:main}
  \noindent {\bf Initialization}: Choose $x_0\in \cX$, $\bs y_0\in \bs\cY$. \\ For all $i \in \{1,\dots,n\}$, set $w_0^{(i)} = \sum_{j\in
J(i)}M_{j,i}^\star\,\bs y_0^{(j)}(i)$.\\ For all $j \in \{1,\dots,p\}$, set
$z_0^{(j)} = \frac 1{m_j}\sum_{i\in I(j)}\bs y_0^{(j)}(i)$.

  \noindent {\bf Iteration $k$}: Define:
  \begin{align*}
    \overline y_{k+1} &= \prox_{\sigma,h^\star}\big(z_k+D(\sigma) Mx_k\big) \\
    \overline x_{k+1} &= \prox_{\tau,g}\Big(x_k-D(\tau)\left(\nabla f(x_k)+2M^\star \overline {y}_{k+1}- w_k\right)\Big)\,.
  \end{align*}
  For $i=i_{k+1}$ and for each \revisiontrois{ $(l,j) \in \mathcal J(i_{k+1})$}, update:
  \begin{align*}
    &{x}_{k+1}^{(i)} = \overline {x}_{k+1}^{(i)} \\
    &{\bs y}_{k+1}^{(j)}(l) = {\bs y}_{k}^{(j)}(l) + \revisiontrois{\bs \pi_j(l)} (\overline {y}_{k+1}^{(j)} - {\bs y}_{k}^{(j)}(l))\\
    &w_{k+1}^{(l)} = w_{k}^{(l)} + \sum_{(l,j) \in \mathcal J(i)}M_{j,l}^\star\,({\bs y}_{k+1}^{(j)}(l)-\bs y_k^{(j)}(l))\\
    &z_{k+1}^{(j)} = z_{k}^{(j)} + \frac 1{m_j} \sum_{l:(l,j) \in \mathcal J(i)}({\bs y}_{k+1}^{(j)}(l)
    - \bs y_k^{(j)}(l)) \,.
  \end{align*}
  Otherwise, set ${x}_{k+1}^{(i')}=x_k^{(i')}$,
  $w_{k+1}^{(l')}=w_k^{(l')}$, $z_{k+1}^{(j')}=z_k^{(j')}$ and ${\bs y}_{k+1}^{(j')}(l')={\bs y}_{k}^{(j')}(l')$.
\end{algorithm}

\begin{theorem}
  \label{the:main}
Let Assumption~\ref{hyp:main} hold true and
\begin{equation}
\label{eq:tau_algo2}
\tau_i < \frac 1{\beta_i+\rho\left(\sum_{j\in J(i)}(2-\bs \pi_{j}(i))m_j\sigma_j M_{j,i}^\star M_{j,i}\right)}\,.
\end{equation}
Suppose that Eq. (\ref{eq:constraintJ}) holds and that $\cS\neq\emptyset$. Let $(x_k, \bs y_k)$ 
be a sequence generated by Algorithm~\ref{algo:main}. Almost surely,
there exists $(x_*,y_*)\in \cS$ such that
\begin{align*}
  &\lim_{k\to\infty} x_k = x_* \\
  &\lim_{k\to\infty} \bs y_k^{(j)}(i) = y_*^{(j)}\qquad (\forall j \in \{1,\dots,p\},\ \forall i\in I(j))\,.
\end{align*}
\end{theorem}

\subsection{Special Cases}

\revisiontrois{
\subsubsection{The case $\mathcal J(i) = \{i\} \times J(i)$ for all $i$}
\label{sec:Jminimal}

According to \eqref{eq:constraintJ}, the smallest possible choice
for $\mathcal J(i)$ is $\mathcal J(i) = \{i\} \times J(i)$.
In that case, $\bs \pi_j(i) = 1$ for all $j \in J(i)$ and the update of the dual variable simplifies to:
\[
\forall j \in J(i_{k+1}), \quad \bs y_{k+1}^{(j)}(i_{k+1}) = \overline y_{k+1}^j\,.
\]
This choice of dual sampling also implies that the primal and dual variables are grouped into $n$ disjoint 
primal-dual blocks of the type $(x^{(i)}, (\bs y_{k+1}^{(j)}(i))_{j \in J(i)} )$.

\subsubsection{The case $\mathcal J(i) = \cup_{j \in J(i)} I(j) \times J(i)$ for all $i$}
With this update scheme for dual variables, given $i_{k+1}$, we update
$\bs y_{k+1}^{(j)}(l)$ for all $j \in J(i_{k+1})$ and all $l \in I(j)$.
Said otherwise, we update all the copies of $y_{k+1}^{(j)}$ as soon as
one of them has to be updated.

We have $\bs \pi_j(l) = \frac{1}{|I(j)|} = \frac{1}{m_j}$ for all $l \in I(j)$.
The advantage of this update scheme is that, provided there exists $y'$ such that $\bs y_{0}^{(j)}(l) = {y'}_0^{(j)}$ for all $l \in I(j)$, we have for all $l \in I(j)$ and all $k \geq 0$, 
\[
\bs y_{k+1}^{(j)}(l) = {y'}_{k+1}^{(j)}= \frac{1}{m_j} \overline y^{(j)}_{k+1} + (1-\frac{1}{m_j}) {y'}_k^{(j)}.
\]
Hence, choosing $\mathcal J(i) = \cup_{j \in J(i)} I(j) \times J(i)$
allows us to undo the duplication of dual variables and reduce the size of the vector of dual variables from 
the number of nonzero elements in $M$, $\mathrm{nnz}(M)$, to its number of rows $p$.

This shows the following equivalence result.
\begin{proposition}
Algorithm~\ref{algo:notduplicated} with initial point $y_0'$ is equivalent to Algorithm~\ref{algo:main} with the choice of dual sampling $\mathcal J(i) = \cup_{j \in J(i)} I(j) \times J(i)$, $\forall i \in \{1, \ldots , n \}$ and initial point 
$\bs y_0^{(j)}(l) = {y_0'}^{(j)}$, $\forall j \in \{1, \ldots, p\}$, $\forall l \in I(j)$.
\end{proposition}
So, a byproduct of the proof of Theorem~\ref{the:main} will be a proof for Theorem~\ref{algo:notduplicated}.

}

\subsubsection{The Case $m_1=\dots=m_p=1$}

We consider the special case $m_1=\cdots = m_p = 1$.
\revisiontrois{Otherwise stated, the linear operator $M$ has a single nonzero component $M_{j,i}$ per row $j \in \{1,\dots,p\}$.
This happens for instance in the context of distributed optimization \cite{bianchi2014stochastic}.
This case will also be extensively used in the proofs.

In this scenario, the notations can be drastically simplied. Indeed, for every $j\in \{1,\dots,p\}$,
$I(j)$ is a singleton. The corresponding set of duplicated dual variables $(\bs y_k^{(j)}(i):i\in I(j))$ is reduced
to a single variable $\bs y_k^{(j)}(I(j))$, which we shall simply denote as $y_k^{(j)}$.
According to~(\ref{eq:constraintJ}), $\mathcal J(i)$ is a subset of $\{(l,j) : l\in I(j)\}$ which simply coincides
with the set $\{(I(j),j):j\in\{1,\dots,p\}$. Therefore, the set $\mathcal J(i)$ is uniquely determined by its projection
onto the second set of indices. Otherwise stated, the selection of $\mathcal J(i)$ for a given $i$ is equivalent to the selection
of a subset of $\{1,\dots,p\}$ which we abusively denote by $\mathcal J(i)$ in this paragraph.

Then, Algorithm~\ref{algo:main} simplifies to Algorithm~\ref{algo:casSeparable} below. Note that Algorithm~\ref{algo:casSeparable} has a range of applicability which is different from Algorithm~\ref{algo:notduplicated}. We make an additional assumption on $M$ but we have more freedom on the dual sampling $\mathcal J$.}
\begin{algorithm}
\caption{Coordinate-descent primal-dual algorithm - Case $m_1=\cdots = m_p = 1$.}
\label{algo:casSeparable}
  \noindent {\bf Initialization}: Choose $x_0\in \cX$, $y_0\in \cY$.

  \noindent {\bf Iteration $k$}: Define:
  \begin{align*}
    \overline y_{k+1} &= \prox_{\sigma,h^\star}\big(y_k+D(\sigma) Mx_k\big) \\
    \overline x_{k+1} &= \prox_{\tau,g}\Big(x_k-D(\tau)\big(\nabla f(x_k)+M^\star (2\overline {y}_{k+1}-y_k)\big)\Big)\,.
  \end{align*}
  For $i=i_{k+1}$ and for each $j\in \mathcal J(i_{k+1}) $, update:
  \begin{align*}
    &{x}_{k+1}^{(i)} = \overline {x}_{k+1}^{(i)} \\
    & y_{k+1}^{(j)} = y_{k}^{(j)} + \pi_j (\overline {y}_{k+1}^{(j)} - y_{k}^{(j)}) \,.
  \end{align*}
  Otherwise, set ${x}_{k+1}^{(i')}=x_k^{(i')}$, $y_{k+1}^{(j')}={y}_{k}^{(j')}$.
\end{algorithm}

\subsubsection{The Case $h=0$}

Instanciating Algorithm~\ref{algo:main} in the special case $h=0$, it boils down to the following CD forward-backward algorithm:
\begin{equation}
x_{k+1}^{(i)} = \left\{
  \begin{array}[h]{ll}
    \prox_{\tau,g}^{(i)}\big(x_k-D(\tau)\nabla f(x_k)\big), \quad& \text{if } i=i_{k+1}, \\
    x_k^{(i)}, & \text{otherwise.}
  \end{array}\right.\label{eq:fbCD}
\end{equation}
As a consequence, Algorithm~\ref{algo:main} allows to recover the CD proximal gradient algorithm 
of~\cite{RT:UCDC} with the notable difference that we do \emph{not} assume the separability of $g$.
On the other hand, Assumption~\ref{hyp:main}(\ref{hyp:main-norm}) becomes
$\tau_i < 1/\beta_i$ whereas in the separable case, \cite{RT:UCDC} assumes $\tau_i = 1/\beta_i$.
This remark leads us to conjecture that, even though Assumption~\ref{hyp:main}(\ref{hyp:main-norm}) generally allows for the use of larger
step sizes \revision{than} the ones suggested by the approach of \cite{combettes2014stochastic,bianchi2014stochastic}, 
one might be able to use even larger step sizes than the ones allowed by Theorem~\ref{the:main}.

Note that a similar CD forward-backward algorithm can be found in~\cite{combettes2014stochastic} with no need to require the separability of $g$.
However, the algorithm of~\cite{combettes2014stochastic} assumes that the step size $\tau_i$ (there assumed to be independent of $i$)
is less than $2/\beta$ where $\beta$ is the \emph{global} Lipschitz constant of $\nabla f$. As discussed in the introduction, an attractive feature
of our algorithm is the fact that our convergence condition $\tau_i<1/\beta_i$ only involves the coordinate-wise Lipschitz constant of $\nabla f$.

\subsection{Failure of Stochastic Fej\'er Monotonicity}

As discussed in the introduction, an existing approach to prove convergence
of CD algorithm in a general setting (that is, not restricted to $h=0$ and separable $g$)
is to establish the stochastic Fej\'er monotonicity of the iterates.
The idea was used in~\cite{iutzeler2013asynchronous} and 
extended by~\cite{combettes2014stochastic} and~\cite{bianchi2014stochastic} to a more general setting.
Unfortunately, this approach implies to select a ``small'' step size as noticed in the previous section.
The use of small step size is unfortunate in practice, as it may significantly affect the 
convergence rate. 

It is natural to ask whether the existing convergence proof
based on stochastic Fej\'er monotonicity can be extended to the use of larger step sizes. 
The answer is negative, as shown by the following example.

\begin{example}
Consider the toy problem 
\[
\min_{x \in \mathbb{R}^3} \frac{1}{2} (x^{(1)}+x^{(2)}+x^{(3)}-1)^2
\]
that is we take $f(x) = \frac{1}{2} (x^{(1)}+x^{(2)}+x^{(3)}-1)^2$ and $g=h=M=0$. 
One of the minimizers is $x_* = ( \frac{1}{3},  \frac{1}{3},  \frac{1}{3})$.
The global Lipschitz constant of $\nabla f$ is equal to $3$ and the coordinate-wise Lipschitz constants are equal to 1. 
The CD proximal gradient algorithm (\ref{eq:fbCD}) writes
$$
x_{k+1}^{(i)} = \left\{
  \begin{array}[h]{ll}
    x_k^{(i)}-\tau (x_k^{(1)}+x_k^{(2)}+x_k^{(3)}-1) \quad& \text{if } i=i_{k+1} \\
    x_k^{(i)} & \text{otherwise}
  \end{array}\right.
$$
where we used $\tau_1=\tau_2 =\tau_3\triangleq \tau$ for simplicity.
By Theorem~\ref{the:main}, $x_k$ converges almost surely to $x_*$ whenever $\tau<1$.
Setting $x_0=0$, one has  $\|x_0-x_*\|^2 = \frac{1}{3}$. It is immediately seen that 
$\bE\|x_1-x_*\|^2 = (\tau-\frac 13)^2 + \frac 19+\frac 19$ where $\bE$ represents the expectation. In particular,
$\bE\|x_1-x_*\|^2> \|x_0-x_*\|^2$ as soon as $\tau>2/3$.
Therefore, the sequence $\bE\|x_k-x_*\|^2$ is not decreasing.
This example shows that the proof techniques based on monotone operators and Fej\'er monotonicity
are not directly applicable in the case of long step sizes.
Indeed, as shown in Lemma~\ref{lem:inegSk} below, one needs to make use of another Lyapunov function,
defined in \eqref{eq:contraction}. That inequality shows that the sequence
exhibits a stochastic monotonicity property in the Bregman divergence sense~\cite{bauschke2003bregman}.
\end{example}

\section{Proof of Theorem~\ref{the:main}}

\label{sec:proof}


\subsection{Preliminary Lemma}

For every $(x,y)\in\cX\times \cY$, we define 
\begin{equation}
\label{eq:def-V}
V(x,y) := \frac 12\|x\|^2_{\tau^{-1}}+\ps{y,Mx} +\frac 12\|y\|^2_{\sigma^{-1}}\,.
\end{equation}

\begin{lemma}
\label{lem:lyap}
Let Assumption~\ref{hyp:main}(\ref{hyp:main-convex}-\ref{hyp:main-diff}) hold true.  Let $(x,y)\in \cX\times\cY$ and $(x_*,y_*)\in \cS$. Define
  \begin{align*}
    \overline y &= \prox_{\sigma, h^\star}\big(y+D(\sigma) Mx\big)\\
    \overline x &= \prox_{\tau,g}\Big(x-D(\tau)\big(\nabla f(x)+
    M^\star (2\overline y-y)\big) \Big)
  \end{align*}
and set 
$z=(x,y)$, $z_*=(x_*,y_*)$, $\overline z=(\overline x,\overline y)$. Then,
$$
 \ps{\nabla f(x_*)-\nabla f(x),x_*-\overline x}+V(\overline z-z)\leq  V(z-z_*)-V(\overline z-z_*)\,.
$$
\end{lemma}
\begin{proof}
  The inclusions~(\ref{eq:pd}) also read
  \begin{align*}
    &\forall u\in \cX,\  g(u)\geq g(x_*)+\ps{-\nabla f(x_*)-M^\star y_*,u-x_*}\\
    &\forall v\in \cY,\ h^\star(v)\geq h^\star(y_*)+\ps{Mx_*,v-y_*}\,.
  \end{align*}
  Setting $u=\overline x$ and $v=\overline y$ in the above
  inequalities, we obtain
  \begin{align}
    &  g(\overline x)\geq g(x_*)+\ps{\nabla f(x_*)+M^\star y_*,x_*-\overline x} \label{eq:star-x}\\
    & h^\star(\overline y)\geq h^\star(y_*)+\ps{Mx_*,\overline
      y-y_*}\,. \label{eq:star-y}
  \end{align}
By definition of the proximal operator,
  \begin{align}
    \overline y &= \arg\min_{v\in \cY} h^\star(v) - \ps{v,Mx}+\frac 12\|v-y\|^2_{\sigma^{-1}} \label{eq:bar-y} \\
    \overline x &= \arg\min_{u\in \cX} g(u) + \ps{u,\nabla
      f(x)+M^\star(2\overline y-y)}+\frac
    12\|u-x\|^2_{\tau^{-1}}\,. \label{eq:bar-x}
  \end{align}
  Consider Equality~(\ref{eq:bar-y}) above. It classically implies~\cite{tseng2008accelerated} that for any $v\in
  \cY$,
  \begin{equation}
  \label{eq:prox_hstar_ineq}
    h^\star(\overline y) - \ps{\overline y,Mx}+\frac 12\|\overline y-y\|^2_{\sigma^{-1}} \leq 
    h^\star(v) - \ps{v,Mx}+\frac 12\|v-y\|^2_{\sigma^{-1}}-\frac 12\|\overline y-v\|^2_{\sigma^{-1}}\,.
  \end{equation}
  Setting $v=y_*$, we obtain
\begin{equation}
\label{eq:ybar_leq_ystar}
h^\star(\overline y) \leq h^\star(y_*) + \ps{\overline y-y_*,Mx}+\frac
12\|y_*-y\|^2_{\sigma^{-1}}-\frac 12\|\overline
y-y_*\|^2_{\sigma^{-1}} - \frac 12\|\overline y-y\|^2_{\sigma^{-1}}
\end{equation}
and using~(\ref{eq:star-y}), we finally have
\begin{equation}
  \ps{M(x_*-x),\overline y-y_*}  \leq 
  \frac 12\|y_*-y\|^2_{\sigma^{-1}}-\frac 12\|\overline y-y_*\|^2_{\sigma^{-1}}  - \frac 12\|\overline y-y\|^2_{\sigma^{-1}}
  \label{eq:tmp-y}
\end{equation}
Similarly, Equality~(\ref{eq:bar-x}) implies that for any $u\in
\cX$,
\begin{multline}
\label{eq:prox_g_ineq}
  g(\overline x) + \ps{\overline x,\nabla f(x)+M^\star(2\overline y-y)}+\frac 12\|\overline x-x\|^2_{\tau^{-1}} \\
  \leq g(u) + \ps{u,\nabla f(x)+M^\star(2\overline y-y)}+\frac 12\|u-x\|^2_{\tau^{-1}}-\frac 12\|\overline x-u\|^2_{\tau^{-1}}.
\end{multline}
We set $u=x_*$. This yields
\begin{multline*}
  g(\overline x) \leq g(x_*) + \ps{x_*-\overline x,\nabla
    f(x)+M^\star(2\overline y-y)}+\frac
  12\|x_*-x\|^2_{\tau^{-1}} 
  -\frac 12\|\overline x-x_*\|^2_{\tau^{-1}}
  -\frac 12\|\overline x-x\|^2_{\tau^{-1}}\,.
\end{multline*}
Using moreover Inequality~(\ref{eq:star-x}), we obtain
\begin{multline*}
  \ps{\nabla f(x_*)+M^\star y_*,x_*-\overline x} \
  \leq \ps{x_*-\overline x,\nabla f(x)+M^\star(2\overline y-y)} 
  +\frac 12\|x_*-x\|^2_{\tau^{-1}}-\frac 12\|\overline x-x_*\|^2_{\tau^{-1}} -\frac 12\|\overline x-x\|^2_{\tau^{-1}}
\end{multline*}
hence, rearranging the terms,
\begin{multline*}
  \ps{\nabla f(x_*)-\nabla f(x),x_*-\overline x}-\frac 12\|x_*-x\|^2_{\tau^{-1}}+\frac 12\|\overline x-x_*\|^2_{\tau^{-1}} +\frac 12\|\overline x-x\|^2_{\tau^{-1}} 
  \leq \ps{2\overline y-y-y_*,M(x_*-\overline x)}\,.
\end{multline*}
Summing the above inequality with~(\ref{eq:tmp-y}),
\begin{align*}
  \ps{\nabla f(x_*)-\nabla f(x),&x_*-\overline x}+\frac 12\|\overline
  x-x\|^2_{\tau^{-1}}+\ps{\overline y-y,M(\overline x- x)} +\frac
  12\|\overline y-y\|^2_{\sigma^{-1}}
  \\
  \leq &\frac 12\|x-x_*\|^2_{\tau^{-1}}+ \ps{y-y_*,M(x-x_*)} + \frac
  12\|y-y_*\|^2_{\sigma^{-1}}  \\
  &\qquad-\frac 12\|\overline
  x-x_*\|^2_{\tau^{-1}} - \ps{\overline y-y_*,M(\overline x-x_*)}
  -\frac 12\|\overline y-y_*\|^2_{\sigma^{-1}}\,.
\end{align*}
This completes the proof of the lemma thanks to the definition of $V$.
\end{proof}

\subsection{Study of Algorithm~\ref{algo:casSeparable}}
\label{sec:separable}

We first prove Theorem~\ref{the:main} in the special case $m_1=\cdots = m_p = 1$.
In that case, Algorithm~\ref{algo:main} boils down to Algorithm~\ref{algo:casSeparable}.
We recall that in this case, the vector $\bs y_k^{(j)}$ is reduced to a single value $\bs y_k^{(j)}(i)\in \cY_j$ where $i$
is the unique index such that $M_{j,i}\neq 0$. We simply denote this value by $y_k^{(j)}$.

We denote by $\cF_k$ the filtration generated by the random variable (r.v.) $i_1,\cdots,i_k$.
We denote by $\bE_k(\,.\,)=\bE(\,.\,|\cF_k)$ the conditional expectation w.r.t. $\cF_k$.
\begin{lemma}
\label{lem:esperances}
  Let Assumptions~\ref{hyp:main}(\ref{hyp:main-convex},\ref{hyp:main-diff},\ref{hyp:main-unif}) hold true. Suppose $m_1=\cdots = m_p = 1$.
 Consider Algorithm~\ref{algo:casSeparable} and let $\gamma_1,\dots,\gamma_n, \revision{\gamma'_1, \dots, \gamma'_p}$ be arbitrary positive coefficients.
For every $k\geq 1$ and every  $\cF_k$-measurable pair of random variables $(X,Y)$ on $\cX\times \cY$,  
\begin{align*}
 & \bE_k(x_{k+1}) = \frac 1n \overline x_{k+1} + (1-\frac 1n)x_k\\
  &\bE_k(\|x_{k+1}-X\|^2_{\gamma}) = \frac 1n \|\overline x_{k+1}-X\|^2_{\gamma}+(1-\frac 1n)\|x_{k}-X\|^2_{\gamma}\\
&\bE_k(\|y_{k+1}-Y\|^2_{\gamma'}) =\frac 1n \|\overline y_{k+1}-Y\|^2_{\gamma'}+ (1-\frac 1n)\|y_{k}-Y\|^2_{\gamma'} \revisiontrois{- \frac{1}{n} \|\overline y_{k+1}-y_k\|^2_{D(1-\pi)\gamma'}}\\
&\bE_k (\ps{y_{k+1}-Y,M(x_{k+1}-X)}) =  \frac 1n \ps{\overline y_{k+1}-Y,M(\overline x_{k+1}-X)} \\ 
&+ (1-\frac 1n)\ps{y_k-Y,M(x_k-X)} 
\revisiontrois{-  \frac{1}{n}\ps{D(1-\pi)(\overline y_{k+1}-y_k),M(\overline x_{k+1}-x_k)}}\,.
\end{align*}
\end{lemma}
\begin{proof}
  The first equality is immediate. 

Consider the second one.
$\bE_k(\|x_{k+1}-X\|^2_{\gamma}) = \sum_{i=1}^n \gamma_i\bE_k(\|x_{k+1}^{(i)}-X^{(i)}\|^2)$ which coincides
with $\sum_{i=1}^n \gamma_i(\frac 1n\|\overline x_{k+1}^{(i)}-X^{(i)}\|^2 + (1-\frac 1n)\|x_{k}^{(i)}-X^{(i)}\|^2)$
and the second equality is proved.

Similarly for the third equality,
$\bE_k(\|y_{k+1}-Y\|^2_{\gamma'}) = \sum_{j=1}^p \gamma'_j\bE_k(\|y_{k+1}^{(j)}-Y^{(j)}\|^2)$ and for every $j$,
\begin{align*}
\bE_k(\|y_{k+1}^{(j)}-Y^{(j)}\|^2) = \|y_k^{(j)} + &\pi_j (\overline y_{k+1}^{(j)} - y_k^{(j)})-Y^{(j)}\|^2 \bP(j\in \mathcal J(i_{k+1})) \\
&+\|y_{k}^{(j)}-Y^{(j)}\|^2 \bP(j\notin \mathcal J(i_{k+1})).
\end{align*}
As $j \in J(i_{k+1}) \Leftrightarrow i_{k+1} \in I(j)$, we get
\[
\bP(j\in J(i_{k+1})) = \bP(i_{k+1}\in I(j)) = \mathrm{card}(I(j))/n=1/n.
\]
From \eqref{eq:def_pi_j},
\[
\pi_j = \bP(j \in J(i_{k+1}) | j \in \mathcal J(i_{k+1})) = \frac{\bP(j \in J(i_{k+1}) \,\&\, j \in \mathcal J(i_{k+1}))}{\bP(j \in \mathcal J(i_{k+1}))} = \frac{\bP(j \in J(i_{k+1}))}{\bP(j \in \mathcal J(i_{k+1}))}
\]
and so 
\[
\bP(j \in \mathcal J(i_{k+1})) = \frac{1}{n \pi_j} = \frac{|\{i \,:\, j \in \mathcal J(i)\}|}{n}.
\]
We also have 
\begin{multline*}
\|y_k^{(j)} + \pi_j (\overline y_{k+1}^{(j)} - y_k^{(j)})-Y^{(j)}\|^2 
 = \pi_j \| \overline y_{k+1}^{(j)} - Y^{(j)}\|^2 + (1-\pi_j)  \|  y_k^{(j)} - Y^{(j)}\|^2 - \pi_j (1-\pi_j) \| \overline y_{k+1}^{(j)} - y_k^{(j)} \|^2
\end{multline*}
This leads to
\[
\bE_k(\|y_{k+1}^{(j)}-Y^{(j)}\|^2) =\frac 1n \|\overline y_{k+1}^{(j)}-Y^{(j)}\|^2+ (1-\frac 1n)\|y_{k}^{(j)}-Y^{(j)}\|^2 - \frac{1 - \pi_j}{n} \| \overline y_{k+1}^{(j)} - y_k^{(j)} \|^2.
\]
\revision{This proves the third equality.}

Consider the fourth equality.
Note that
\[
\ps{y_{k+1}-Y,M(x_{k+1}-X)} = \sum_{i=1}^n\sum_{j\in J(i)} \ps{y_{k+1}^{(j)}-Y^{(j)}, M_{j,i}(x_{k+1}^{(i)}-X^{(i)})}.
\]
For any pair $(i,j)$ such that $j\in J(i)$, the conditional expectation of each term in the sum is equal to
\begin{align*}
&\frac {1}n\ps{\pi_j\overline y_{k+1}^{(j)} + (1-\pi_j) y_k^{(j)}-Y^{(j)}, M_{j,i}(\overline x_{k+1}^{(i)}-X^{(i)})}
\\
&\quad + (\frac {1}{n\pi_j} - \frac 1n)\ps{\pi_j\overline y_{k+1}^{(j)} + (1-\pi_j) y_k^{(j)}-Y^{(j)}, M_{j,i}(x_{k}^{(i)}-X^{(i)})} \\
&\quad + (1 -\frac{1}{n\pi_j})\ps{ y_k^{(j)}-Y^{(j)}, M_{j,i}(x_{k}^{(i)}-X^{(i)})}\\
&= \frac {\pi_j}n\ps{\overline y_{k+1}^{(j)}-Y^{(j)}, M_{j,i}(\overline x_{k+1}^{(i)}-X^{(i)})}
+ (1-\frac 2n + \frac {\pi_j}n)\ps{y_{k}^{(j)}-Y^{(j)}, M_{j,i}(x_{k}^{(i)}-X^{(i)})} \\
&\; + (\frac 1n -\frac{\pi_j}{n})\ps{ y_k^{(j)}\!-\!Y^{(j)}, M_{j,i}(\overline x_{k+1}^{(i)}\!-\!X^{(i)})}+(\frac 1n -\frac{\pi_j}{n})\ps{ \overline y_{k+1}^{(j)}\!-\!Y^{(j)}, M_{j,i}(x_k^{(i)}\!-\!X^{(i)})} \\
& = \frac 1n \ps{\overline y_{k+1}^{(j)}-Y^{(j)}, M_{j,i}(\overline x_{k+1}^{(i)}-X^{(i)})}
+ (1-\frac 1n)\ps{y_{k}^{(j)}-Y^{(j)}, M_{j,i}(x_{k}^{(i)}-X^{(i)})}\\
&\; + (\frac 1n \!-\! \frac{\pi_j}{n})\ps{ y_k^{(j)}\!-\!\overline y_{k+1}^{(j)}, M_{j,i}(\overline x_{k+1}^{(i)}\!-\!X^{(i)})}+\gamma_j'(\frac 1n \!-\! \frac{\pi_j}{n})\ps{ \overline y_{k+1}^{(j)}\!-\!y_k^{(j)}, M_{j,i}(x_k^{(i)}\!-\!X^{(i)})} \\
& =  \frac 1n \ps{\overline y_{k+1}^{(j)}-Y^{(j)}, M_{j,i}(\overline x_{k+1}^{(i)}-X^{(i)})}
+ (1-\frac 1n)\ps{y_{k}^{(j)}-Y^{(j)}, M_{j,i}(x_{k}^{(i)}-X^{(i)})}\\
&\quad + (\frac 1n -\frac{\pi_j}{n})\ps{ y_k^{(j)}-\overline y_{k+1}^{(j)}, M_{j,i}(\overline x_{k+1}^{(i)}-x_k^{(i)})}
\end{align*}
Finally, we obtain
\begin{multline*}
\bE(\ps{y_{k+1}-Y,M(x_{k+1}-X)}) = \frac 1n \ps{\overline y_{k+1}-Y,M(\overline x_{k+1}-X)} \\ + (1-\frac 1n)\ps{y_k-Y,M(x_k-X)} 
-  \frac{1}{n}\ps{D(1-\pi)(\overline y_{k+1}-y_k),M(\overline x_{k+1}-x_k)}
\end{multline*}
which in turn implies the fourth equality in the Lemma.
\end{proof}

Assume that $\tau_i^{-1}>\beta_i$ for each $i \in \{1,\dots,n\}$. Define for every $z = (x,y)\in\cX\times \cY$,
\begin{align}
\label{eq:def-Vtilde}
\tilde V(z) = \tilde V(x,y) &:= \frac 12\|x\|^2_{\tau^{-1}-\beta}+\ps{D(2-\pi)y,Mx} +\frac 12\|y\|^2_{\sigma^{-1}(2-\pi)}
\,.
\end{align}
\begin{lemma}
\label{lem:inegSk} 
   Let Assumptions~\ref{hyp:main}(\ref{hyp:main-convex},\ref{hyp:main-diff},\ref{hyp:main-smooth},\ref{hyp:main-unif}) hold true. Suppose $m_1=\cdots = m_p = 1$ and assume that
$\tau^{-1}_i>\beta_i$ for each $i\in \{1,\dots,n\}$. Consider Algorithm~\ref{algo:casSeparable} and define for every $k\in \bN$,
\begin{equation}
S_{k,*} := f(x_k)-f(x_*)-\ps{\nabla f(x_*),x_k-x_*}\,.
\label{eq:Sk}
\end{equation}
Then the following inequality holds:
\begin{align}
  \label{eq:contraction}
  \bE_k\left[ S_{k+1,*}+V(z_{k+1}-z_*)\right] \leq (1-\frac 1n) S_{k,*}+V(z_{k}-z_*)-\frac 1n \tilde V(\overline z_{k+1}-z_k)
\end{align}
where $\overline z_{k+1} = (\overline x_{k+1},\overline y_{k+1})$.
\end{lemma}

\begin{proof}
We can write the relations of Lemma~\ref{lem:esperances} as
\begin{align*}
& \|\overline x_{k+1}-X\|^2_{\tau^{-1}} = n \bE_k(\|x_{k+1}-X\|^2_{\tau^{-1}}) - (n-1)\|x_{k}-X\|^2_{\tau^{-1}}\\
&\|\overline y_{k+1}-Y\|^2_{\sigma^{-1}} = n \bE_k(\|y_{k+1}-Y\|^2_{\sigma^{-1}}) - (n-1)\|y_{k}-Y\|^2_{\sigma^{-1}} + \|\overline y_{k+1}-y_k\|^2_{\sigma^{-1}(1-\pi)}\\
&\ps{\overline y_{k+1}-Y,M(\overline x_{k+1}-X)} = n \bE_k (\ps{y_{k+1}-Y,M(x_{k+1}-X)})    \\
& \qquad\qquad - (n-1)\ps{y_{k}-Y,M(x_{k}-X)} + \ps{D(1-\pi)(\overline y_{k+1}-y_k),M(\overline x_{k+1}-x_k)} \,.
\end{align*}

Choosing $Z=(X,Y)$, denoting
$z_k=(x_k,y_k)$ and $\overline z_k=(\overline x_k,\overline y_k)$,
we obtain 
\begin{multline}
V(\overline z_{k+1}-Z) = n\bE_k (V(z_{k+1}-Z))-n V(z_{k}-Z) + V(z_k-Z) \\
+ \frac 12 \|\overline y_{k+1} - y_k\|^2_{\sigma^{-1}(1-\pi)} + 
\ps{D(1-\pi)(\overline y_{k+1}-y_k),M(\overline x_{k+1}-x_k)} \,.\label{eq:EV}
\end{multline}
We shall denote
\begin{equation}
\label{eq:def-Rpi}
R_\pi = \frac 12 \|\overline y_{k+1} - y_k\|^2_{\sigma^{-1}(1-\pi)} + 
\ps{D(1-\pi)(\overline y_{k+1}-y_k),M(\overline x_{k+1}-x_k)}
\end{equation}
Let $z_* = (x_*,y_*)\in \cS$.  By Lemma~\ref{lem:lyap},
$$
\ps{\nabla f(x_*)-\nabla f(x_k),x_*-\overline x_{k+1}}+V(\overline
z_{k+1}-z_k) \leq V(z_k-z_*)-V(\overline z_{k+1}-z_*)\,.
$$
Identifying $Z$ in~(\ref{eq:EV}) to $z_*$ and $z_k$ successively, we
obtain
\begin{align*}
\ps{\nabla f(x_*)-\nabla f&(x_k),x_*-\overline x_{k+1}}+n\bE_k
(V(z_{k+1}-z_k)) 
\leq nV(z_{k}-z_*)- n\bE_k (V(z_{k+1}-z_*)) - 2 R_\pi
\end{align*}
Dividing both sides of the above inequality by $n$ and using that
$\overline x_{k+1} = n\bE_k(x_{k+1}) - (n-1)x_k$, we obtain
\begin{multline*}
\ps{\nabla f(x_*)-\nabla f(x_k),x_*-\bE_k(x_{k+1}) + (1-\frac
  1n)(x_k-x_*)}+\bE_k (V(z_{k+1}-z_k)) \\
   \leq V(z_{k}-z_*)- \bE_k
(V(z_{k+1}-z_*)) - \frac{2}{n} R_\pi\,.
\end{multline*}
Rearranging the terms,
\begin{align}
\label{eq:bound_with_gradients}
  \bE_k&\left[\ps{\nabla f(x_k)-\nabla
      f(x_*),x_{k+1}-x_k}+V(z_{k+1}-z_*)\right] \\ &\leq -\frac 1n
  \ps{\nabla f(x_k)-\nabla f(x_*),x_k-x_*}+V(z_{k}-z_*)-\bE_k
  (V(z_{k+1}-z_k)) -\frac{2}{n}  R_\pi \notag 
\end{align}
We now use Assumption~\ref{hyp:main}(\ref{hyp:main-smooth}), knowing
that $x_{k+1}$ only differs from $x_k$ along coordinate $i_{k+1}$
\begin{align}
\label{eq:taylor}
f(x_{k+1}) &\leq f(x_k)+\ps{\nabla f(x_k),x_{k+1}-x_k} +
\frac{\beta_{i_{k+1}}}2\|x_{k+1}-x_k\|^2 \notag \\
& = f(x_k)+\ps{\nabla f(x_k),x_{k+1}-x_k} +
\frac 12 \|x_{k+1}-x_k\|_\beta^2 
\end{align}
which implies that $\ps{\nabla f(x_k),x_{k+1}-x_k}\geq
f(x_{k+1})-f(x_k)-\frac 12\|x_{k+1}-x_k\|^2_\beta$. Thus, plugging 
this into \eqref{eq:bound_with_gradients},
\begin{align*}
  \bE_k\left[
    f(x_{k+1})-f(x_k)-\frac 12\|x_{k+1}-x_k\|^2_\beta-\ps{\nabla
      f(x_*),x_{k+1}-x_k}+V(z_{k+1}-z_*)\right] \\ \leq -\frac 1n
  \ps{\nabla f(x_k)-\nabla f(x_*),x_k-x_*}+V (z_{k}-z_*)-\bE_k
  (V(z_{k+1}-z_k)) - \frac{2}{n}  R_\pi\,.
\end{align*}
Introducing the quantity $S_{k, *}$ as in~(\ref{eq:Sk}), the inequality
simplifies to
\begin{align*}
  \bE_k\Big[ S_{k+1,*}+V(&z_{k+1}-z_*)
    -\frac 12\|x_{k+1}-x_k\|_\beta^2\Big] \\
    \leq & f(x_k)
  -f(x_*)-(1-\frac 1n)\ps{\nabla f(x_*),x_k-x_*} -\frac 1n \ps{\nabla
    f(x_k),x_k-x_*} \\
    &\qquad +V(z_{k}-z_*)-\bE_k (V(z_{k+1}-z_k)) - \frac{2}{n} R_\pi\,.
\end{align*}
An estimate of the right-hand side is obtained upon noticing that
$\ps{\nabla f(x_k),x_k-x_*}\geq f(x_k)-f(x_*)$. Therefore,
\begin{align*}
\bE_k\big[
  S_{k+1,*}+V(&z_{k+1}-z_*)-\frac 12\|x_{k+1}-x_k\|_\beta^2\big] 
\leq 
(1-\frac 1n)S_{k,*}+V(z_{k}-z_*)-\bE_k (V(z_{k+1}-z_k)) - \frac{2}{n}  R_\pi\,.
\end{align*}
Using Lemma~\ref{lem:esperances}, \eqref{eq:def-Vtilde} and \eqref{eq:def-Rpi}, it is immediate that 
\begin{align*}
\bE_k (V (z_{k+1}-z_k)& - \frac 12\|x_{k+1}-x_k\|_\beta^2) + \frac{2}{n}  R_\pi\\
&=\frac 1n V(\overline z_{k+1}-z_k)- \frac{1}{n} R_\pi 
- \frac{1}{2n}\norm{\overline x_{k+1}^{(j)} - x_k}^2_\beta + \frac{2}{n}  R_\pi \\
&= \frac 1n\tilde V(\overline z_{k+1}-z_k)
\end{align*} 
and the proof is complete.
\end{proof}

Recall that we denote by $\rho(A)$ the spectral radius of a matrix $A$.
\begin{lemma}
\label{lem:Vtilde-positive}
Suppose that $m_1=\cdots=m_p=1$ and assume that the following condition holds for every $i \in \{1,\dots,n\}$:
  \begin{equation}
    \label{eq:tau-separable}
    \tau_i < \frac 1{\beta_i+\rho\left(\sum_{j\in J(i)}(2-\pi_j)\sigma_j M_{j,i}^\star M_{j,i}\right)}\,.
  \end{equation}
Then $\tilde V^{1/2}$ is a norm on $\cX\times\cY$. 
\end{lemma}
Note that under the assumptions of Lemma~\ref{lem:Vtilde-positive},
$V^{1/2}$ is also, \emph{a fortiori}, a norm, but that $V^{1/2}$ need not be a norm.
\begin{proof}
  Let $\gamma^{-1}=\tau^{-1}-\beta$. Denote by $\sigma_j' = (2-\pi_j)\sigma_j$ for all $j$ and by $D(\sigma')$ the
  diagonal matrix on $\cY\to\cY$ defined by $D(\sigma')(y) :=
  (\sigma'_1y^{(1)},\dots,\sigma'_py^{(p)})$ for every
  $y=(y^{(1)},\dots,y^{(p)})$. We define $D(\gamma)$ similarly on
  $\cX\to\cX$. By \cite[Theorem 7.7.6]{horn2012matrix}, a sufficient (and
  necessary) condition for $\tilde V$ to be a squared norm is that
  $D(\gamma^{-1})\succ M^\star D(\sigma') M$ (where notation $A \succ
  B$ means that $A-B$ is a positive definite matrix).  Defining
  $R=D({\sigma'}^{1/2})M D(\gamma^{1/2})$ (that is, $R_{j,i} =
  \sqrt{\gamma_i\sigma_j'}M_{j,i}$ for every $j,i$), the condition
  reads equivalently $\rho(R^\star R) <1$. As the set $I(j)$ is
  reduced to a unique element for all $j$, the matrix $R^\star R$ is
  (block) diagonal. Precisely, for any $1\leq i,\ell\leq n$, the
  $(i,\ell)$-component $(R^\star R)_{i,\ell}$ is zero whenever $i\neq
  \ell$ and is equal to $ (R^\star R)_{i,i} =
  \gamma_i\sum_{j\in J(i)}\sigma'_j M_{j,i}^\star M_{j,i}$ otherwise.
  The condition $\rho(R^\star R) <1$ yields $\gamma_i\rho\left(\sum_{j\in
      J(i)}\sigma_j' M_{j,i}^\star M_{j,i}\right)<1$ for each
  $i \in \{1,\dots,n\}$ which is in turn equivalent to \eqref{eq:tau-separable}.
\end{proof}

\begin{proof}[Proof of Theorem~1 in the case $m_1 = \ldots = m_p =1$]
Let $z_*$ be an arbitrary point in $\cS$.  Whenever
condition~(\ref{eq:tau-separable}) is met, the r.v. $V(z_k-z_*)$ and
$\tilde V(\overline z_{k+1}-z_k)$ are non-negative.  The
r.v. $S_{k,*}$ is non-negative as well by convexity of $f$. We review two
important consequences of Lemma~\ref{lem:inegSk}.

$\bullet$ Define $U_k:=S_{k,*} + V(z_{k}-z_*)$. A first consequence of Lemma~\ref{lem:inegSk} is that for all~$k$,
$$
\bE_k(U_{k+1})\leq U_{k}-\frac 1n S_{k,*}\,.
$$
Recalling that $U_k$ and $S_k$ are non-negative r.v., the Robbins-Siegmund Lemma \cite{robbins1971convergence}
implies that almost surely,
$\lim_{k\to\infty}U_k$ exists and $\sum_k S_{k,*}<\infty$. In particular, $S_{k,*}$ converges almost surely to zero.
By definition of $U_k$, this implies that $\lim_{k\to\infty}V(z_{k}-z_*)$ exists almost surely.
Following the argument of \cite[Prop. 9]{bertsekas2011incremental} (see also \cite{iutzeler2013asynchronous}, \cite[Prop. 2.3]{combettes2014stochastic}),
this implies that there exists an event $A$ of probability one such that 
for every $\omega\in A$ and every $\check z\in \cS$, $\lim_{k\to\infty} V^{1/2}(z_{k}(\omega)-\check z)$ exists.

$\bullet$ A second consequence of Lemma~\ref{lem:inegSk} is that, by
taking the expectation $\bE$ of both handsides
of~(\ref{eq:contraction}),
$$
\bE\left[ S_{k+1,*}+V(z_{k+1}-z_*)\right] \leq \bE[S_{k,*}+V(z_{k}-z_*)]-\frac 1n \bE(\tilde V(\overline z_{k+1}-z_k))
$$
and by summing these inequalities, we obtain 
\begin{equation}
\label{eq:sum_Vtilde_bounded}
0\leq
S_{0,*}+V(z_{0}-z_*) - \frac 1n \sum_{i=0}^{k} \bE(\tilde V(\overline
z_{i+1}-z_i)).
\end{equation}
  Thus $\bE(\sum_{i=0}^{\infty} \tilde V(\overline
z_{i+1}-z_i))<\infty$. The integrand is non-negative by
Lemma~\ref{lem:Vtilde-positive}.  It is therefore finite almost
everywhere.  In particular, the sequence $\tilde V(\overline
z_{k+1}-z_k)$ converges almost surely to zero. By
Lemma~\ref{lem:Vtilde-positive}, $\overline z_{k+1}-z_k$ converges to
zero almost surely. Say $\overline z_{k+1}(\omega)-z_k(\omega)\to 0$
for every $\omega\in B$ where $B$ is a probability event of
probability one.

We introduce the mapping $T:\cX\times\cY\to\cX\times\cY$ such that for any $(x,y)\in \cX\times\cY$,
the quantity $T(x,y)$ coincides with the couple $(\overline  x,\overline y)$ given by
  \begin{align*}
    \overline y &= \prox_{\sigma, h^\star}\big(y+D(\sigma) Mx\big)\\
    \overline x &= \prox_{\tau,g}\big(x-D(\tau)\nabla f(x)-D(\tau)
    M^\star (2\overline y-y)\big) \,.
  \end{align*}
With this definition, $\overline z_{k+1} =T(z_k)$. By non-expansiveness of the proximity operator, it
is straightforward to show that $T$ is continuous. It is also straightforward to verify that its
set of fixed points coincides with~$\cS$.

From now on to the end of this paragraph, we select a fixed $\omega\in A\cap B$.
Note that $z_{k}(\omega)$ is a bounded sequence. Let $\tilde z$ be a cluster point of the latter.
We have shown that $T(z_{k}(\omega))-z_k(\omega)\to 0$ which implies that $T(\tilde z)-\tilde z=0$ by
continuity of $T$. Thus, $\tilde z\in \cS$. This implies that
$\lim_{k\to\infty} V^{1/2}(z_{k}(\omega)-\tilde z)$ exists.
Since $V^{1/2}(z_{k}(\omega)-\tilde z)$ tends to zero at least on some subsequence, we conclude that 
$\lim_{k\to\infty} V^{1/2}(z_{k}(\omega)-\tilde z)=0$. Otherwise stated, the sequence $z_{k}(\omega)$ converges
to some point $\tilde z\in \cS$.
This completes the proof of Theorem~\ref{the:main} in the case $m_1=\dots=m_p=1$.
\end{proof}

\subsection{General Case}
\label{sec:general_case}

For every $j \in \{1,\dots,p\}$, $\bs\cY_j=\cY_j^{I(j)}$ is equipped with the 
inner product $\ps{\bs u,\bs v} = \sum_{i\in I(j)}\ps{\bs u(i),\bs v(i)}$.
\revision{The space $\bs\cY_j$ stores $I(j)$ duplicates of the original problem's $j$th dual variable $y_j$.}
We introduce the averaging operator $S_j:\bs\cY_j\to \cY_j$  defined for every ${\bs u}\in\bs\cY_j$ by
$$
S_j({\bs u}) := \frac 1{m_j}\sum_{i\in I(j)}{\bs u}(i)\,.
$$
\revision{The averaging operators allows us to come back from duplicated dual variables to actual dual variables.}
For any $u\in \cY_j$, we denote by ${\bs 1}_{m_j}\otimes u = (u,\dots,u)$ the vector of $\bs\cY_j$ whose components all coincide with~$u$.

We introduce the linear operator $K_j:\cX\to \bs\cY_j$ by 
$$
K_j(x) = (M_{j,i}(x^{(i)})\,:\,i\in I(j))
$$
The operators $S:\bs\cY\to\cY$, 
$K:\cX\to\bs \cY$ are respectively defined by
$S(\bs y) := (S_1(\bs y^{(1)}),\dots,S_p({\bs y}^{(p)}))$
and $K(x) := (K_1(x),\dots,K_p(x))$. 
It is immediate to verify that
\begin{equation}
M = D(m)SK
\label{eq:M=mSK}
\end{equation}
where $m=(m_1,\dots,m_p)$.
In order to have some insights, the following example illustrates the construction of $K$ for a given $M$.
\begin{example}
  Let $\cX=\cY=\bR^3$ and define $M:\cX\to\cY$ as the $3\times 3$ matrix
$$
M =
\begin{pmatrix}
  M_{1,1} & M_{1,2} & 0 \\
  0 & M_{2,2} & 0 \\
  M_{3,1} & M_{3,2} & M_{3,3}
\end{pmatrix}\,.
$$
Here, $I(1)=\{1,2\}$ is the set of non-zero coefficients of the first row of $M$ and it cardinal is $m_1=2$. Similarly $m_2=1$, $m_3=3$ and $\bs \cY=\bR^6$.
Then $K:\bR^3\to \bR^6$ coincides with the matrix
$$
K =
\begin{pmatrix}
  M_{1,1} & 0 & 0 \\
  0 & M_{1,2} & 0 \\
  0 & M_{2,2} & 0 \\
  M_{3,1} & 0 & 0 \\
  0 & M_{3,2} & 0 \\
  0 & 0 & M_{3,3}
\end{pmatrix}
$$
and each row of $K$ contains exactly one non-zero coefficient.
 On the other hand, $S$ and $D(m)$ respectively coincide with 
$$
S =
\begin{pmatrix}
  \frac 12 & \frac 12 & 0 & 0 & 0 & 0 \\
0 & 0 & 1 &0 &0 &0 \\
0&0&0&\frac 13&\frac 13&\frac 13
\end{pmatrix}
\qquad\text{ and }\qquad 
D(m) =
\begin{pmatrix}
  2 & 0 & 0 \\
  0&1&0\\
  0&0&3
\end{pmatrix}
$$
and obviously $D(m)SK=M$.
\end{example}

We define the function $\overline h := h\circ (D(m)S)\,.$
By~(\ref{eq:M=mSK}), Problem~(\ref{eq:pb}) is equivalent to
\begin{equation}
\min_{x\in \cX} f(x)+g(x)+\overline h(Kx)\,.
\label{eq:pb2}
\end{equation}
We denote by $\bs \cS$ the set of primal-dual solutions of the above problem \emph{i.e.}, the set
of pairs $(x_*,\bs y_*)\in \cX\times\bs\cY$ satisfying
\begin{align*}
  0&\in \nabla f(x_*)+\partial g(x_*)+K^\star \bs y_* \\
  0&\in -Kx_*+\partial \overline h^\star(\bs y_*)\,.
\end{align*}
Substituting $M$ with $K$, we may now apply Algorithm~\ref{algo:casSeparable} to~\eqref{eq:pb2}.
For a fixed parameter $\sigma = (\sigma_1,\dots,\sigma_p)$, we define $\tilde\sigma_j:= m_j\sigma_j$ and we define
$\tilde\sigma\in \bR^{\sum_{j=1}^pm_j}$ as the vector
$\tilde \sigma := (\tilde \sigma_1{\bs 1}_{m_1},\dots, \tilde \sigma_p{\bs 1}_{m_p})$ where $\bs 1_{m_j}$
is a vector of size $m_j$ whose components are all equal to one. Algorithm~\ref{algo:casSeparable} writes

  \noindent {\bf Initialization}: Choose $x_0\in \cX$, $\bs y_0\in \bs\cY$.

  \noindent {\bf Iteration $k$}: Define:
  \begin{align}
    \overline {\bs y}_{k+1} &= \prox_{\tilde \sigma,\overline h^\star}\big(\bs y_k+D(\tilde \sigma) Kx_k\big) \label{eq:algo-brute-1}\\
    \overline x_{k+1} &= \prox_{\tau,g}\Big(x_k-D(\tau)\big(\nabla f(x_k)+K^\star (2\overline {\bs y}_{k+1}-\bs y_k)\big)\Big)\,. \label{eq:algo-brute-2}
  \end{align}
  For $i=i_{k+1}$ and for each $(l,j)\in \mathcal J(i_{k+1}) $, update:
  \begin{align}
    &{x}_{k+1}^{(i)} = \overline {x}_{k+1}^{(i)} \\
    & \bs y_{k+1}^{(j)}(l) = {\bs y}_{k}^{(j)}(l) + \bs \pi_j(l) (\overline  {\bs y}_{k+1}^{(j)}(l) - {\bs y}_{k+1}^{(j)}(l))\,. \label{eq:algo-brute-4}
  \end{align}
  Otherwise, set ${x}_{k+1}^{(i)}=x_k^{(i)}$, $\bs y_{k+1}^{(j)}(l)={\bs y}_{k}^{(j)}(l)$.


Using the result of the Section~\ref{sec:separable} and the properties of $K$, 
the sequence $(x_k,\bs y_k)$ converges almost surely to a primal-dual point of Problem~(\ref{eq:pb2}),
provided that such a point exists and that the following condition holds:
$$
 \tau_i < \frac 1{\beta_i\!+\!\rho\Big(\!\!\!\!\!\!\!\displaystyle\sum_{(l,j)\in \{i\}\times J(i)}\!\!\!\!\!\!\!\!\!(2-\bs \pi_j(l))\tilde\sigma_j K_{(l,j),i}^\star K_{(l,j),i}\Big)} = \frac 1{\beta_i\!+\!\rho\Big(\!\!\displaystyle\sum_{j\in J(i)}\!\!(2-\bs \pi_j(i))\tilde\sigma_j M_{j,i}^\star M_{j,i}\Big)}
$$
which is equivalent to \eqref{eq:tau_algo2}. 
It remains to prove that the algorithm given by the iterations~(\ref{eq:algo-brute-1})--(\ref{eq:algo-brute-4})
coincides with Algorithm~\ref{algo:main}. To that end, we need the following Lemma.
\begin{lemma}
\label{lem:hdiff}
For any $\bs y\in \bs\cY$, 
\[
\prox_{\tilde \sigma,\overline h^\star}(\bs y) = ({\bs 1}_{m_1}\otimes \prox_{\sigma,h^\star}^{(1)}(S(\bs y)),\dots, {\bs 1}_{m_p}\otimes \prox_{\sigma,h^\star}^{(p)}(S(\bs y)))\,.
\]
\end{lemma}
\begin{proof}
We have $\overline h(\bs y) =h(m_1S_1(\bs y^{(1)}),\dots,m_pS_p(\bs y^{(p)}))$. Thus,
$$
\overline h^\star(\bs \varphi) = \sup_{\bs y\in \bs\cY} \ps{\bs\varphi,\bs y} - h(m_1S_1(\bs y^{(1)}),\dots,m_pS_p(\bs y^{(p)}))
$$
For all $j \in \{1,\dots,p\}$, denote by $\bs\cC_j$ the subset of $\bs\cY_j$ formed by the vectors of the form
$(u,\dots,u)$ for some $u\in \cY_j$, and define $\bs\cC = \bs\cC_1\times\cdots\times \bs\cC_p$.
 Clearly, $\overline h^\star(\bs \varphi)=+\infty$ whenever $\bs \varphi\notin \bs\cC$
and $\partial \overline h^\star(\bs\varphi)=\emptyset$ in that case.
If on the other hand $\bs \varphi\in \bs\cC$, one can write 
$\bs\varphi$ under the form $\bs \varphi=({\bs 1}_{m_1}\otimes \varphi^{(1)},\dots,{\bs 1}_{m_p}\otimes \varphi^{(p)})$
for some $\varphi\in \cY$. In that case,
\begin{align*}
  \overline h^\star(\bs \varphi) &= \sup_{y\in \cY}  \sum_{j=1}^p\ps{{\bs 1}_{m_j}\otimes \varphi^{(j)},{\bs 1}_{m_j}\otimes y^{(j)}} - h(m_1y^{(1)},\dots,m_py^{(p)})\\
&= \sup_{y\in \cY}  \sum_{j=1}^p\ps{\varphi^{(j)}, m_jy^{(j)}} - h(m_1y^{(1)},\dots,m_py^{(p)})\ = h^\star(\varphi)\,.
\end{align*}
Then, $\bs u\in \partial\overline h^\star(\bs \varphi)$ if and only if for every $\psi\in \cY$,
$h^\star(\psi)\geq h^\star(\varphi) + \sum_{j=1}^p\ps{\bs u^{(j)},{\bs 1}_{m_j}\otimes(\psi^{(j)}-\varphi^{(j)})}$
or equivalently,
$
h^\star(\psi)\geq h^\star(\varphi) + \sum_{j=1}^p\ps{ m_jS_j(\bs u^{(j)}),\psi^{(j)}-\varphi^{(j)}}\,.
$
Therefore, $\bs u\in \partial\overline h^\star(\bs \varphi)$ if and only if 
$D(m)S(\bs u) \in \partial h^\star(\varphi)$.

Now consider an arbitrary $\bs y\in \bs\cY$ and set 
 $\bs q=\prox_{\tilde \sigma,\overline h^\star}(\bs y)$.
This is equivalent to 
\begin{equation}
D(\tilde\sigma^{-1})(\bs y -\bs q)\in \partial \overline h^\star(\bs q).\label{eq:prox-op}
\end{equation}
In particular, $\bs q\in \dom(\partial {\overline h}^\star)$ and thus $\bs q$ has the form $\bs q =({\bs 1}_{m_1}\otimes q^{(1)},\dots, {\bs 1}_{m_p}\otimes q^{(p)})$
for some $q\in \cY$. The inclusion~(\ref{eq:prox-op}) reads 
$D(m)SD(\tilde\sigma^{-1})(\bs y- \bs q))\in \partial  h^\star(q)$. 
Since $D(m)SD(\tilde\sigma^{-1})=D(\sigma^{-1})S$, we obtain 
$D(\sigma^{-1})(S(\bs y)-q)\in \partial  h^\star(q)$ which is equivalent to
$q=\prox_{\sigma,h^\star}(S(\bs y))$. This completes the proof.
\end{proof}

The proof of the following Lemma is immediate.
\begin{lemma}
\label{lem:adjointK}
For any $\bs y\in \bs\cY$, 
\[
K^\star(\bs y) = (\sum_{j\in J(1)} M_{j1}^\star (\bs y^{(j)}(1)),\dots,\sum_{j\in J(n)} M_{jn}^\star (\bs y^{(j)}(n))).
\]
In particular, for any $y\in \cY$,
$$
K^\star({\bs 1}_{m_1}\otimes y^{(1)},\dots,{\bs 1}_{m_p}\otimes y^{(p)} ) = M^\star y\,.
$$
\end{lemma}

\revisiontrois{
The following example shows how we are going to use the concept of duplication.
\begin{example}[Total variation]

Let us consider $\cX = \mathbb R^{n_1 \times n_2 \times n_3}$, $\cY = \mathbb R^{3 \times n_1 \times n_2 \times n_3}$ 
and the total variation regularizer defined as $h \circ M$ where
\[
h(y) = \sum_{i_1=1}^{n_1}  \sum_{i_2=1}^{n_2} \sum_{i_3=1}^{n_3} \sqrt{\sum_{j=1}^3 y_{j,i_1,i_2,i_3}^2}= \sum_{i_1=1}^{n_1}  \sum_{i_2=1}^{n_2} \sum_{i_3=1}^{n_3}  h_{i_1,i_2,i_3}(y_{:,i_1,i_2,i_3})
\]
 and $M$ defined by blocks of
the type
\begin{align*}
M_{(i_1,i_2,i_3)} =
\begin{pmatrix}
  (x_{i_1,i_2,i_3}) & (x_{i_1+1,i_2,i_3}) & (x_{i_1,i_2+1,i_3}) & (x_{i_1,i_2,i_3+1})\\
  -1 & 1 & 0 & 0 \\
  -1 & 0 & 1 & 0 \\
  -1 & 0 & 0 & 1
\end{pmatrix} &
\begin{matrix}
\\
(y_{1, i_1,i_2,i_3})\\
(y_{2, i_1,i_2,i_3})\\
(y_{3, i_1,i_2,i_3})\\
\end{matrix}
\end{align*}
Each line has two nonzero elements so we duplicate dual variables as
\begin{align*}
K_{(i_1,i_2,i_3)} =
\begin{pmatrix}
  (x_{i_1,i_2,i_3}) & (x_{i_1+1,i_2,i_3}) & (x_{i_1,i_2+1,i_3}) & (x_{i_1,i_2,i_3+1}) \\
  -1 & 0 & 0 & 0 \\
   0 & 1 & 0 & 0 \\
  -1 & 0 & 0 & 0 \\
   0 & 0 & 1 & 0 \\
  -1 & 0 & 0 & 0 \\
   0 & 0 & 0 & 1
\end{pmatrix}&
\begin{matrix}
\\
(\bs y_{1, i_1,i_2,i_3}(1))\\
(\bs y_{1, i_1,i_2,i_3}(2))\\
(\bs y_{2, i_1,i_2,i_3}(1))\\
(\bs y_{2, i_1,i_2,i_3}(2))\\
(\bs y_{3, i_1,i_2,i_3}(1))\\
(\bs y_{3, i_1,i_2,i_3}(2))\\
\end{matrix}
\end{align*}
Hence, we cam write 
$\bar{h}_{i_1,i_2,i_3}(\bs y_{i_1,i_2,i_3,:}) =\sqrt{\sum_{j=1}^3(\bs y_{i_1,i_2,i_3,j}(1) +\bs y_{i_1,i_2,i_3,j}(2) )^2}$

$
\prox_{m\sigma,\overline h^*}(\boldsymbol y) = ({\boldsymbol 1}_{m_1}\otimes \prox_{\sigma,h^*}^{(1)}(S(\boldsymbol y)),\dots, {\boldsymbol 1}_{m_p}\otimes \prox_{\sigma,h^*}^{(p)}(S(\boldsymbol y)))
$
becomes, denoting $e_l$ the $l$th coordinate vector, 
\[
\prox_{2\sigma,\overline h^*_{i_1,i_2,i_3}}(\boldsymbol y_{i_1,i_2,i_3,:}) = \begin{pmatrix}
e_1^\top \prox_{\sigma, h_{i_1,i_2,i_3}^*}(\bs y_{i_1,i_2,i_3,:}(1) + \bs y_{i_1,i_2,i_3,:}(2)) \\[-0.3ex]
e_1^\top \prox_{\sigma, h_{i_1,i_2,i_3}^*}(\bs y_{i_1,i_2,i_3,:}(1) + \bs y_{i_1,i_2,i_3,:}(2)) \\[-0.3ex]
e_2^\top \prox_{\sigma, h_{i_1,i_2,i_3}^*}(\bs y_{i_1,i_2,i_3,:}(1) + \bs y_{i_1,i_2,i_3,:}(2)) \\[-0.3ex]
e_2^\top \prox_{\sigma, h_{i_1,i_2,i_3}^*}(\bs y_{i_1,i_2,i_3,:}(1) + \bs y_{i_1,i_2,i_3,:}(2)) \\[-0.3ex]
e_3^\top \prox_{\sigma, h_{i_1,i_2,i_3}^*}(\bs y_{i_1,i_2,i_3,:}(1) + \bs y_{i_1,i_2,i_3,:}(2)) \\[-0.3ex]
e_3^\top \prox_{\sigma, h_{i_1,i_2,i_3}^*}(\bs y_{i_1,i_2,i_3,:}(1) + \bs y_{i_1,i_2,i_3,:}(2)) 
\end{pmatrix}
\]

Suppose we would like to update $x_{3,4,5}$:
\begin{itemize}
\item The dual variables corresponding to $x_{3,4,5}$ are
$\bs y_{3,4,5,1}(1)$, $\bs y_{3,4,5,2}(1)$, $\bs y_{3,4,5,3}(1)$, 
$\bs y_{2,4,5,1}(2)$, $\bs y_{3,3,5,2}(2)$ and $\bs y_{3,4,4,3}(2).$

\item We compute 
$\prox_{2\sigma,\overline h^*_{3,4,5}}(\boldsymbol y_{3,4,5,:})$,
$\prox_{2\sigma,\overline h^*_{2,4,5}}(\boldsymbol y_{2,4,5,:})$,
$\prox_{2\sigma,\overline h^*_{3,3,5}}(\boldsymbol y_{3,3,5,:})$ and
$\prox_{2\sigma,\overline h^*_{3,4,4}}(\boldsymbol y_{3,4,4,:})$,
which amounts to 12 real numbers.

\item We update only the 6 useful dual values.

\end{itemize}
\end{example}
}

We are now in a position to simplify the iterations~(\ref{eq:algo-brute-1})--(\ref{eq:algo-brute-4}).
For every $k$, we define \revision{the vectors} $\overline y_{k+1} = \prox_{\sigma,h^\star}\big(S(\bs y_k+D(\tilde \sigma) Kx_k)\big)$ and
$\overline {\bs y}_{k+1} = ({\bs 1}_{m_1}\otimes\overline y_{k+1}^{(1)},\dots, {\bs 1}_{m_p}\otimes \overline y_{k+1}^{(p)})$.
Upon noting that $SD(\tilde \sigma) K=D(\sigma)D(m)SK=D(\sigma)M$, we obtain 
\begin{equation}
  \overline y_{k+1} = \prox_{\sigma,h^\star}(z_k+D(\sigma)Mx_k) \label{eq:simp-ybar} 
\end{equation}
where we defined $z_k = S(\bs y_k)$, otherwise stated, for each $j\in \{1,\dots,p\}$,
$$
z_k^{(j)} = \frac 1{m_j}\sum_{i\in I(j)}\bs y_k^{(j)}(i)\,.
$$
Note that $z_{k+1}$ differs from $z_k$ only along the components $j$ for which $\bs y_{k+1}^{(j)}(i)$ differs from $\bs y_{k}^{(j)}(i)$
for some~$i$. That is, $z_{k+1}^{(j)}=z_k^{(j)}$ for each $j$ such that $(i,j)\notin \mathcal J(i_{k+1})$ for all $I$ while for any $j$ such that there exists $i$ such that $(i,j)\in \mathcal J(i_{k+1})$,
\begin{equation}
z_{k+1}^{(j)} =z_{k}^{(j)}+\frac {1}{m_j}\sum_{i:(i,j) \in \mathcal J(i_{k+1})}(\bs y_{k+1}^{(j)}(i)-\bs y_k^{(j)}(i))\,.
\label{eq:simp-z}
\end{equation}
Now consider equation~(\ref{eq:algo-brute-2}). By Lemma~\ref{lem:adjointK}, $K^\star\overline {\bs y}_{k+1}=M^\star\overline y_{k+1}$. Thus,
setting $w_k = K^\star\bs y_k$, equation~(\ref{eq:algo-brute-2}) simplifies to:
\begin{equation}
\overline x_{k+1} = \prox_{\tau,g}\Big(x_k-D(\tau)\big(\nabla f(x_k)+ (2M^\star\overline {y}_{k+1}-w_k)\big)\Big)\,.\label{eq:simp-xbar}
\end{equation}
By Lemma~\ref{lem:adjointK} again, $w_k = (\sum_{j\in J(1)} M_{j1}^\star \bs y_k^{(j)}(1),\dots,\sum_{j\in J(n)} M_{jn}^\star \bs y_k^{(j)}(n))$.
Therefore, $w_{k+1}$ only differs from $w_k$ along the coordinates $i$ such that there exists $(i,j) \in \mathcal J(i_{k+1})$ and the update reads:
\begin{equation}
  \label{eq:simp-w}
  w_{k+1}^{(i)} = w_{k}^{(i)} + \sum_{(i,j) \in \mathcal J(i_{k+1})} M_{j,i}^\star (\bs y_{k+1}^{(j)}(i)-\bs y_k^{(j)}(i))\,.
\end{equation}
Putting all pieces together, the update equations~(\ref{eq:simp-ybar})--(\ref{eq:simp-w}) coincide with Algorithm~\ref{algo:main}.
We have thus proved that Algorithm~\ref{algo:main} is such that $(x_k,\bs y_k)$ converges to a primal-dual point of Problem~(\ref{eq:pb2})
provided that such a point exists.
To complete the proof, the final step is to relate the primal-dual solutions of Problem~(\ref{eq:pb2}) 
to the primal-dual solutions of the initial Problem~(\ref{eq:pb}).

Consider the mapping $G:\cX\times\cY\to \cX\times \bs\cY$ defined by
\[
G(x,y) := (x,({\bs 1}_{m_1}\otimes y^{(1)},\dots,{\bs 1}_{m_p}\otimes y^{(p)})).
\]
\begin{lemma}
  $\bs\cS = G(\cS)$.
\end{lemma}
\begin{proof}
Let $(x,y)\in \cX\times\cY$ and set $\bs y= ({\bs 1}_{m_1}\otimes y^{(1)},\dots,{\bs 1}_{m_p}\otimes y^{(p)})$.
Then  $M^\star y = K^\star \bs y$, therefore
$$
0\in \nabla f(x)+\partial g(x) + K^\star \bs y\ \Leftrightarrow\ 0\in \nabla f(x)+\partial g(x) + M^\star y\,.
$$
Moreover,
   \begin{eqnarray*}
    0 \in -Kx+\partial \overline h^\star(\bs y) &\Leftrightarrow&     Kx\in \partial \overline h^\star(\bs y) \\
&\Leftrightarrow&     D(m)S(Kx)\in \partial  h^\star(y) \\
&\Leftrightarrow&     Mx\in \partial  h^\star(y) 
  \end{eqnarray*}
where we used Lemma~\ref{lem:hdiff} along with the identities $D(m)SK=M$ and $S(\bs y) = y$. The proof is completed upon noting that if 
$(x,\bs y)\in \bs\cS$, then there exists $y\in \cY$ such that $\bs y$ has the form $\bs y= ({\bs 1}_{m_1}\otimes y^{(1)},\dots,{\bs 1}_{m_p}\otimes y^{(p)})$
.
\end{proof}
We have shown that, almost surely, $(x_k,\bs y_k)$ converges to some point in $G(\cS)$.
This completes the proof of Theorem~\ref{the:main}.

\revisiontrois{

\section{Convergence rate}
\label{sec:rate}

In this section, we are interested in the rate of convergence
of the method. 
We consider three cases:
\begin{itemize}
\item $h$ is Lipschitz continuous: we prove a $O(1/\sqrt k)$ decrease for the function value (Theorem~\ref{the:sublinear_rate}).
\item $h = I_{\{b\}}$, {\em i.e.} $h(y) = 0$ if $y = b$ and $h(y) = +\infty$ otherwise. This corresponds to an optimization problem under the affine constraints $Mx=b$.
We prove a $O(1/\sqrt k)$ decrease for the function value and the feasibility (Theorem~\ref{the:sublinear_rate}). 
\item $f+g$ is strongly convex and $\nabla h$ is Lipschitz continuous: we prove a $O(e^{-\mu k})$ rate for the distance to the optimum (Theorem~\ref{the:strconv}).
\end{itemize}
These convergence guarantees are of the same order as what can be
obtained by other primal-dual methods like the ADMM~\cite{davis2017faster}, {\em i.e.} $O(1/\sqrt{k})$ in general and
linear rate of convergence under strong convexity assumptions.

\begin{theorem}
\label{the:sublinear_rate}
Define for $\alpha \geq 1$,
\begin{align*}
& C_{1,\alpha} = \max_{1 \leq i \leq n} \frac{\tau_i^{-1} + \tau_i^{-1/2}\rho(\sum_{j \in J(i)}m_j \sigma_j M_{j,i}^\star M_{j,i})^{1/2}}{\tau_i^{-1} - \rho(\sum_{j \in J(i)}m_j \sigma_j M_{j,i}^\star M_{j,i})}(1+\frac{n}{\alpha}) \\
& C_{2, \alpha} = \Big(1 + \max_{1 \leq i \leq n}\frac{\alpha^{-1}(n(n-1)+1)+1}{\tau_i^{-1} - \beta_i - \rho(\sum_{j \in J(i)}(2-\bs \pi_j(i))m_j\sigma_j M_{j,i}^\star M_{j,i})}\beta_i\Big)\,.
\end{align*}
We have that $C_{1,\alpha}$ and $C_{2,\alpha}$ are nonincreasing with respect to $\alpha$, and thus bounded.

Define the number of iterations $\hat K \in \{1, \ldots, k\}$ as a random variable, independent of $\{i_1, \ldots, i_{k}\}$ and such that $\Pr(\hat K = l) = \frac{1}{k}$ for all $l \in \{1, \ldots, k\}$.

If $h$ is $L(h)$-Lipschitz in the norm $\norm{\cdot}_{D(m)\sigma}$, then for all $k \geq 0$, 
\begin{multline*}
\bE(f(\bar x_{\hat K})+g(\bar x_{\hat K}) + h(M\bar x_{\hat K})  - f(x_*) - g(x_*) - h(Mx_*) ) 
\leq \frac{C_{2,\sqrt k}+2 C_{1,k}}{\sqrt k}  n (S_{0,*}+V(z_{0}-z_*))
+ \frac{4}{\sqrt{k}}L(h)^2 
\end{multline*}
where $V$ is defined in \eqref{eq:def-V} and $S_{0, *}$ is defined in \eqref{eq:Sk}.

If $h = I_{\{b\}}$, then for all $k \geq 0$, 
\begin{align*}
\bE(f(\bar x_{\hat K}) + g(\bar x_{\hat K}) &- f(x_*) - g(x_*)) 
\leq \frac{C_{2,\sqrt k}+2 C_{1,k}}{\sqrt k}  n \big(S_{0,*}+V(z_{0}-z_*)\big) + \norm{y_*}\bE(\norm{ M \bar x_{\hat K} - b })\\
\bE(\norm{M \bar x_{\hat K} - b}_{D(m)\sigma}) &\leq \frac{2}{\sqrt{k}} \Big(\sqrt{C_{2, \sqrt k}+2C_{1,k}}+\sqrt{2C_{1,k}}\Big) \big(n (S_{0,*}+V(z_{0}-z_*))  \big)^{1/2}\big)
\end{align*}
\end{theorem}
\begin{proof}
We begin with the proof for Algorithm~\ref{algo:casSeparable}, that is the case $m_1 = \ldots = m_p = 1$.

We combine the following inequalities proved in the previous sections and that are valid for all $(x,y) \in \cX \times \cY$.
\begin{align*}
& g(\overline x_{k+1}) + \ps{\overline x_{k+1},\nabla f(x_k)+M^\star(2\overline y_{k+1}-y_k)}+\frac 12\|\overline x_{k+1}-x_k\|^2_{\tau^{-1}} \\
&\qquad  \overset{\eqref{eq:prox_g_ineq}}{\leq} g(x) + \ps{x,\nabla f(x_k)+M^\star(2\overline y_{k+1}-y_k)}+\frac 12\|x-x_k\|^2_{\tau^{-1}}-\frac 12\|\overline x_{k+1}-x\|^2_{\tau^{-1}} \\
& h^\star(\overline y_{k+1}) - \ps{\overline y_{k+1},Mx_k}+\frac 12\|\overline y_{k+1}-y_k\|^2_{\sigma^{-1}} 
\overset{\eqref{eq:prox_hstar_ineq}}{\leq} 
h^\star(y) - \ps{y,Mx_k}+\frac 12\|y-y_k\|^2_{\sigma^{-1}}-\frac 12\|\overline y_{k+1}-y\|^2_{\sigma^{-1}} \\
& \bE_k(f(x_{k+1})) \overset{\eqref{eq:taylor}+Lem.~\ref{lem:esperances}}{\leq} f(x_k)+ \frac 1n \ps{\nabla f(x_k),\bar x_{k+1}-x_k} +
\frac{1}{2n}\|\bar x_{k+1}-x_k\|^2_\beta \\
& f(x) \geq f(x_k) + \ps{\nabla f(x_k), x - x_k}
\end{align*}
We obtain that for all $z \in \cX \times \cY$ such that $z$ is measurable with respect to $\cF_k$, 
\begin{multline*}
g(\bar x_{k+1}) + n \bE_k(f(x_{k+1})) - (n-1) f(x_k) + \ps{M \bar x_{k+1}, y} - h^\star(y) 
+ h^\star(\bar y_{k+1}) \\ - \ps{M^\top \bar y_{k+1}, x} - g(x) - f(x)
\leq V(z_k - z) - V(\bar z_{k+1} - z) - V(\bar z_{k+1} - z_k) + \frac{1}{2}\|\bar x_{k+1}-x_k\|^2_\beta
\end{multline*}
As $\nabla f$ is $n$-Lipschitz in the norm $\norm{\cdot}_\beta$~\cite{RT:PCDM} and $n \bE(x_{k+1}) - (n-1) x_k - \bar x_{k+1} = 0$,
\begin{align*}
n \bE_k(f(x_{k+1})) &- (n-1) f(x_k) \geq n \bE_k\big(f(\bar x_{k+1}) + \ps{\nabla f(\bar x_{k+1}), x_{k+1} - \bar x_{k+1}} \big) \\
& \qquad - (n-1) \big(f(\bar x_{k+1})+ \ps{\nabla f(\bar x_{k+1}), x_{k} - \bar x_{k+1}} + \frac{n}{2} \norm{x_k - \bar x_{k+1}}^2_\beta\big) \\
& \geq f(\bar x_{k+1}) - \frac{n(n-1)}{2}\norm{x_k - \bar x_{k+1}}^2_\beta \;.
\end{align*}
We also have for all $\alpha > 0$, 
\begin{align*}
V(z_k - z) &- V(\bar z_{k+1} - z) - V(\bar z_{k+1} - z_k) = \ps{z_k - \bar z_{k+1}, \bar z_{k+1} - z}_V \\
&\leq 2 V(z_k - \bar z_{k+1})^{1/2} V(\bar z_{k+1} - z)^{1/2} \leq 
\alpha V(z_k - \bar z_{k+1}) + \frac{1}{\alpha} V(\bar z_{k+1} - z)
\end{align*}
Gathering everything, we get
\begin{multline*}
g(\bar x_{k+1}) + f(\bar x_{k+1}) + \ps{M \bar x_{k+1}, y} - h^\star(y) 
+ h^\star(\bar y_{k+1}) - \ps{M^\top \bar y_{k+1}, x}\\  - g(x) - f(x)
- \frac{1}{\alpha} V(\bar z_{k+1} - z)\leq \alpha V(z_k - \bar z_{k+1})  + \frac{n(n-1)+1}{2}\|\bar x_{k+1}-x_k\|^2_\beta
\end{multline*}
We can show by tedious but straightforward algebra that the norms $V^{1/2}$, $\tilde V^{1/2}$ and $(1/2(\norm{x}^2_{\tau^{-1}}+\norm{y}^2_{\sigma^{-1}}))^{1/2}$ are equivalent with constants given by
\begin{align*}
V(z) & \leq \Big(\max_{1 \leq i \leq n} 1 + \sqrt{\tau_i \rho(\sum_{j \in J(i)} \sigma_j M_{j,i}^\star M_{j,i})}\Big)\frac{1}{2} (\norm{x}^2_{\tau^{-1}} +\norm{y}^2_{\sigma^{-1}}) 
\leq 2 \times \frac{1}{2} (\norm{x}^2_{\tau^{-1}}+\norm{y}^2_{\sigma^{-1}}) \\
\frac{1}{2} (&\norm{x}^2_{\tau^{-1}}+\norm{y}^2_{\sigma^{-1}}) \leq 
\max_{1 \leq i \leq n} \frac{\tau_i^{-1} + \tau_i^{-1/2}\rho(\sum_{j \in J(i)}\sigma_j^{1/2} M_{j,i}^\star M_{j,i})^{1/2}}{\tau_i^{-1} - \rho(\sum_{j \in J(i)}\sigma_j M_{j,i}^\star M_{j,i})} 
V(z) 
 = C_{1, \infty} V(z) \\
\alpha V&(z) + \frac{n(n-1)+1}{2}\|\bar x_{k+1}-x_k\|^2_\beta \\
&\leq \Big(\alpha + \max_{1 \leq i \leq n}\frac{n(n-1)+1+\alpha}{\tau_i^{-1} - \beta_i -  \rho(\sum_{j \in J(i)}(2-\pi_j)\sigma_jM_{j,i}^\star M_{j,i})}\beta_i\Big) \tilde V(z) =  \alpha C_{2,\alpha} \tilde V(z) 
\end{align*}
where $C_{2,\alpha} \in O(1)$ for $\alpha \to \infty$.
Denoting the smoothed gap~\cite{tran2015smooth} as
\begin{multline*}
\mathcal G_{\frac{2}{\alpha}}(\bar z_{k}, \bar z_{k}) = \sup_{z} g(\bar x_{k}) + f(\bar x_{k}) + \ps{M \bar x_{k}, y} - h^\star(y) 
+ h^\star(\bar y_{k}) \\ - \ps{M^\top \bar y_{k}, x} - g(x) - f(x)
- \frac{2}{2\alpha} \norm{\bar x_{k} - x}^2_{\tau^{-1}} - \frac{2}{2\alpha} \norm{\bar y_{k} - y}^2_{\sigma^{-1}},
\end{multline*}
we have
\begin{align*}
\mathcal G_{\frac{2}{\alpha}}(\bar z_{k}, \bar z_{k})  \leq \alpha C_{2,\alpha} \tilde V(\bar z_{k}-z_{k-1})
\end{align*}
Now, by \eqref{eq:sum_Vtilde_bounded} and the fact that $\hat K$ is independent of the coordinate selection process,
\[
\bE(\tilde V(\overline
z_{\hat K}-z_{\hat K-1})) \leq \sum_{i=1}^{k} \frac {1}{k} \bE(\tilde V(\overline
z_{i}-z_{i-1}))\overset{\eqref{eq:sum_Vtilde_bounded}}{\leq} \frac{n}{k} (S_{0,*}+V(z_{0}-z_*))
\]
so 
\begin{align*}
\bE(\mathcal G_{\frac{2}{\alpha}}(\bar z_{\hat K}, \bar z_{\hat K}))  \leq \frac{\alpha C_{2,\alpha}}{k}  n (S_{0,*}+V(z_{0}-z_*))
\end{align*}
Taking $\alpha = \sqrt k$ as in~\cite{davis2017faster}, we get
\begin{align*}
\bE(\mathcal G_{\frac{2}{\sqrt k}}(\bar z_{\hat K}, \bar z_{\hat K}))  \leq  \frac{C_{2,\sqrt k}}{\sqrt k}  n (S_{0,*}+V(z_{0}-z_*))
\end{align*}

We can also bound
\begin{align*}
\frac 12 \bE(&\norm{\bar x_{\hat K} - x_*}^2_{\tau^{-1}}) \leq C_{1, \infty} \bE(V(\bar z_{\hat K} - z_*)) \\ &\hspace{-1.5em}\overset{\eqref{eq:EV}+\eqref{eq:def-Rpi}}{=}C_{1, \infty} \bE(n V(z_{\hat K}-z_*)-n V(z_{\hat K-1}-z_*) + V(z_{\hat K-1}-z_*) + R_\pi^{(\hat K)})\\
& = \frac{C_{1, \infty}}{k}\! \sum_{i=1}^k \bE(n V(z_{i}-z_*)-n V(z_{i-1}-z_*) + V(z_{i-1}-z_*) + R_\pi^{(i)})\\
& = \frac{C_{1, \infty}}{k} \bE(n V(z_k - z_*) - n V(z_0 - z_*) + \sum_{i=1}^k V(z_{i-1}-z_*) + R_\pi^{(i)}) \\
& \overset{\eqref{eq:contraction}}{\leq}C_{1, \infty}\!\frac{n + k}{k} (S_{0,*} + V(z_0 - z_*)) + 
\frac{C_{1, \infty}}{k}\sum_{i=1}^k  R_\pi^{(i)} - \tilde V(\bar z_i - z_{i-1}) \\
& \leq C_{1,k} (S_{0,*} + V(z_0 - z_*))
\end{align*}
where the last inequality follows from 
$R_\pi^{(i)} - \tilde V(\bar z_i - z_{i-1}) = \frac{1}{2} \norm{\bar x_i - x_{i-1}}^2_{\beta} - V(\bar z_i - z_{i-1}) \leq 0$.

If $h$ is $L(h)$-Lipschitz in the norm $\norm{\cdot}_{\sigma}$, we can choose $y \in \partial h(M\bar x_k) \neq \emptyset$ so that $\langle M \bar x_k, y \rangle - h^*(y) = h(M \bar x_k)$, and 
$x = x^\star$ so that $h^*(\bar y_k) - \langle M^\top \bar y_k, x^\star\rangle \geq -h(M x^\star)$

We then use the inequality
\begin{multline*}
\mathcal G_{\frac{2}{\sqrt{k}}}(\bar z_{\hat K}, \bar z_{\hat K}) \geq f(\bar x_{\hat K})+g(\bar x_{\hat K}) + h(M\bar x_{\hat K}) 
- \frac{4}{\sqrt{k}}L(h)^2 - f(x_*) - g(x_*) - h(Mx_*) - \frac{1}{ \sqrt{k}}\norm{\bar x_{\hat K} - x_*}^2_{\tau^{-1}}
\end{multline*}
to conclude.

If $h = I_{\{b\}}$, then using Lemma 1 in \cite{tran2015smooth}, we get that
\begin{align*}
\bE(f(\bar x_{\hat K}) + g(\bar x_{\hat K}) & - f(x_*) - g(x_*)) \leq \frac{C_{2,\sqrt k}}{\sqrt k}  n (S_{0,*}+V(z_{0}-z_*)) + \\
&\bE(\frac{1}{\sqrt k} \norm{\bar x_{\hat K} - x_*}^2_{\tau^{-1}} + \frac{1}{\sqrt k} \norm{\bar y_{\hat K} - y_*}^2_{\sigma^{-1}} - \ps{y_*, M \bar x_{\hat K} - b })  \\ 
\bE(\norm{M \bar x_{\hat K} - b}_\sigma) &\leq \frac{2}{\sqrt k}\Big(\bE(\norm{\bar y_{\hat K} - y_\star}_{\sigma^{-1}}) + \big[\bE(\norm{\bar y_{\hat K} - y_\star}_{\sigma^{-1}}^2)\\
&+ \frac{2}{2/\sqrt{k}} \frac{C_{2, \sqrt k}}{\sqrt{k}} n (S_{0,*}+V(z_{0}-z_*)) + \bE(\norm{\bar x_{\hat K} - x_\star}^2_{\tau^{-1}}) \big]^{1/2}\Big)
\end{align*}

To obtain the result for Algorithm~\ref{algo:main} we only
need to remark that when we need to duplicate dual variables we have 
$h^\star(\bar y_k) = \overline h^\star(\bar {\bs y}_k)$.
One then just needs to replace $\sigma_j$ by $m_j\sigma_j$ in the
conditions.
\end{proof}
}

\revisiontrois{
\begin{remark}
To prove the result of Theorem~\ref{the:sublinear_rate}, we use a random number of iterations. This has also been proposed for instance in~\cite{Stoch-dual-Coord-Ascent} for the stochastic
dual coordinate ascent algorithm. Note that the number of iterations
can be sampled beforehand, which means that the procedure comes with
no computational cost. When $\hat K$ iterations have taken place,
one just needs to compute $\bar x_{\hat K+1}$ once in order to obtain the guarantee.
\end{remark}	
}

\revisiontrois{

We also have a fast rate if the problem has particular properties.
We prove that if the Lagrangian function satisfies a strong convexity and strong concavity assumption, then Algorithm~\ref{algo:main} converges 
exponentially fast with a rate that depends on the step size.
\begin{assum}
\label{hyp:strconv}
There exists non-negative constants $\mu_g$ and $\mu_f$ such that $\mu_f + \mu_g>0$ and a constant $\mu_{h^\star}>0$ such that  $g$ is $\mu_g$-strongly convex in the norm $\|\cdot \|_{\tau^{-1}}$, $f$ is $\mu_f$-strongly convex in the norm $\|\cdot \|_{\tau^{-1}}$ and $h^\star$ is $\mu_{h^\star}$-strongly convex in the norm $\|\cdot \|_{\sigma^{-1}}$.
\end{assum}

\revisiontrois{
\begin{theorem}
\label{the:strconv}
For $z = (x,\bs y)$, denote 
$V^\mu(z) = V(z) + \mu_g \|x\|_{\tau^{-1}}^2 \!+  \mu_{h^\star}'\|\bs y\|_{(D(m)\sigma)^{-1}}^2$
where 
$\mu_{h^\star}' = \min(\mu_{h^\star}, \sup \{\mu \! > \! 0 : \forall i, \tau_i^{-1} \!>\! \beta_i + \rho(\sum_{j \in J(i)}\!\! \frac{(2-\bs \pi_j(i))^2 \sigma_j m_j}{2-\bs \pi_j(i) - \mu (1-\bs \pi_j(i))} M_{j,i}^\star M_{j,i})\}$
(note that if $\bs \pi_j(i) = 1$ for all $i$ and $j$, then $\mu_{h^\star}' = \mu_{h^\star}$).
If Assumption~\ref{hyp:strconv} holds then the iterates of Algorithm~\ref{algo:main} satisfy
\begin{align*}
\bE\left[ S_{k,*}+V^\mu(z_{k}-z_*) \right] \leq \Big(1 - \frac 1n \frac{(\mu_f + 2 \mu_g)\mu_{h^\star}'}{\mu_f + 2 \mu_g + \mu_{h^\star}'}\Big)^k \left[ S_{0,*}+V^\mu(z_{0}-z_*) \right] \,.
\end{align*}
\end{theorem}
}

In order to prove this theorem, we begin with a lemma that generalizes Lemma~\ref{lem:lyap}.
\begin{lemma}
If Assumption~\ref{hyp:strconv} holds, then
\begin{multline*}
\ps{\nabla f(x_*) -\nabla f(x),x_*-\overline x}+V(\overline z - z)
  \\
  \leq V(z - z_*)  -V(\overline z - z_*) - \mu_g \|\overline
    x-x\|^2_{\tau^{-1}} - \mu_{h^\star}\|\overline y-y\|^2_{\sigma^{-1}}\,.
\end{multline*}
\end{lemma}

\begin{proof}
Assumption~\ref{hyp:strconv} gives us: for $(x_*, y_*) \in \cS$,
\begin{align}
&  g(\overline x)\geq f(x_*)+g(x_*)+\ps{\nabla f(x_*) + M^\star y_*,x_*-\overline x} + \frac{\mu_g}{2}\| x - x_*\|^2_{\tau^{-1}} \, , \label{eq:mux}\\
& h^\star(\overline y)\geq h^\star(y_*)+\ps{Mx_*,
  \overline y-y_*}+ \frac{\mu_{h^\star}}{2}\|\overline y - y_*\|^2_{\sigma^{-1}} \,.
  \label{eq:muy}
\end{align}
With the same argument as in \eqref{eq:ybar_leq_ystar}, we have
$$
h^\star(\overline y) \leq h^\star(y_*) + \ps{\overline y-y_*,Mx}+\frac
12\|y_*-y\|^2_{\sigma^{-1}}-\frac {1+\mu_{h^\star}}{2}\|\overline
y-y_*\|^2_{\sigma^{-1}} - \frac 12\|\overline y-y\|^2_{\sigma^{-1}}
$$
and so using \eqref{eq:muy}
\begin{equation}
  \ps{M(x_*-x),\overline y-y_*}  \leq 
  \frac 12 \|y - y_*\|^2_{\sigma^{-1}} - \frac {1+2\mu_{h^\star}}2 \|\overline y - y_*\|_{\sigma^{-1}}^2 -\frac 12\|\overline y-y\|^2_{\sigma^{-1}} \label{eq:tmp-y-muy}
\end{equation}
Similarly, we have 
\begin{multline*}
\ps{\nabla f(x_*) - \nabla f(x),x_*-\overline x}-\frac{1+2\mu_g}{2} \|x - x_*\|^2_{\tau^{-1}}+\frac 12 \|\overline x - x_*\|^2_{\tau^{-1}} 
+\frac 12\|\overline x-x\|^2_{\tau^{-1}} 
  \leq \ps{2\overline y-y-y_*,M(x_*-\overline x)}\,.
\end{multline*}
Summing the above inequality with~(\ref{eq:tmp-y-muy}),
and recalling the definition of $V(z) = V(x, y) = \frac 12\|x\|^2_{\tau^{-1}}+\ps{y,Mx} +\frac 12\|y\|^2_{\sigma^{-1}}$, we get
\begin{multline*}
\ps{\nabla f(x_*) -\nabla f(x),x_*-\overline x}+V(\overline z - z)
  \leq V(z - z_*)  -V(\overline z - z_*) - \mu_g \|\overline
    x-x\|^2_{\tau^{-1}} - \mu_{h^\star}\|\overline y-y\|^2_{\sigma^{-1}}
\end{multline*}

\end{proof}
}

\revisiontrois{

\begin{proof}[Proof of Theorem \ref{the:strconv}]
We begin with the case $m_1 = \ldots = m_p$.

By Assumption~\ref{hyp:main}\eqref{hyp:main-norm}, if $\mu_{h^\star}>0$, then $\mu_{h^\star}'>0$ and
if $h^\star$ is $\mu_{h^\star}$-strongly convex, it is also $\mu_{h^\star}'$-strongly convex.
Then, by a straightforward adaptation of the proof of Lemma~\ref{lem:inegSk} 
to the strongly convex case,
we have 
\begin{align*}
\bE_k&\left[ S_{k+1,*}+V(z_{k+1}-z_*) + \frac{2\mu_g}{2}\|x_{k+1} - x_*\|_{\tau^{-1}}^2 + \frac{2 \mu_{h^\star}'}{2}\|y_{k+1} - y_*\|_{\sigma^{-1}}^2 \right] \\
&\leq (1-\frac 1n) S_{k,*}+V(z_{k}-z_*)
+ \frac{2(n-1) \mu_g - \mu_f}{2n}\|x_k - x_*\|_{\tau^{-1}}^2 \\
& \qquad + \frac{2 (n-1) \mu_{h^\star}'}{2n}\|y_k - y_*\|_{\sigma^{-1}}^2
-\frac 1n \tilde V(\overline z_{k+1}-z_k) + \frac{\mu_{h^\star}'}{n}\|\overline y_{k+1} - y_k\|^2_{\sigma^{-1}(1-\pi)}
\end{align*}
As soon as  $\tau_i^{-1} > \beta_i+\rho\left(\sum_{j\in J(i)}\frac{(2-\pi_j)^2}{2-\pi_j - \mu_{h^\star}'(1-\pi_j)}\sigma_j M_{j,i}^\star M_{j,i}\right)$, we can remove the term $-\frac 1n \tilde V(\overline z_{k+1}-z_k) + \frac{\mu_{h^\star}'}{n}\|\overline y_{k+1} - y_k\|^2_{\sigma^{-1}(1-\pi)} \leq 0$. This is indeed guaranteed by the definition of $\mu_{h^\star}'$.

In order to prove a linear convergence rate $(1-\eta)$, it suffices to prove
that $(1-\frac 1n) \leq (1-\eta)$ and that with respect to the order of semi-definite matrices,
\[
\begin{bmatrix}
\tau^{-1}(1 + \frac{2(n-1) \mu_g - \mu_f}{n}) \!\!\! &  M^\star \\
M   &  \!\!\! \sigma^{-1} (1 + \frac{2 (n-1) \mu_{h^\star}'}{n})
\end{bmatrix}
\preceq
(1-\eta)
\begin{bmatrix}
\tau^{-1}(1 + 2\mu_g) \!\!\! &  M^\star \\
M   & \!\!\!  \sigma^{-1} (1 + 2 \mu_{h^\star}')
\end{bmatrix}
\]
Using the fact that $M$ is block-diagonal, this gives for all $i$ the conditions
\begin{align*}
& 1 + \frac{2(n-1) \mu_g - \mu_f}{n} \leq (1-\eta) (1 + 2\mu_g) \\
& 1 + \frac{2 (n-1) \mu_{h^\star}'}{n} \leq (1-\eta)(1 + 2 \mu_{h^\star}') \\
& \tau_i^{-1} (-\eta(1 + 2 \mu_g) + \frac{\mu_f+2\mu_g}{n}) \geq \sum_{j \in J(i)} \frac{\sigma_j}{-\eta (1+ 2 \mu_{h^\star}') + \frac{2 \mu_{h^\star}'}{n}} \eta^2 M_{j,i}^\star M_{j,i} \,.
\end{align*}
Using the second condition we can multiply the third one by $-\eta (1+ 2 \mu_{h^\star}') + \frac{2 \mu_{h^\star}'}{n} \geq 0$ and we obtain the condition
\begin{align*}
\eta^2 \Big( &\tau_i^{-1} - \sum_{j \in J(i)} \sigma_j M_{j,i}^\star M_{j,i} \Big) + \tau_i^{-1} \Big(\eta^2 (2\mu_g+2\mu_{h^\star}') \\ 
&- \eta\big(\frac{\mu_f + 2 \mu_g + 2\mu_{h^\star}'}{n} - \frac{4\mu_g \mu_{h^\star}'}{n}-\frac{2\mu_f \mu_{h^\star}' + 4 \mu_g \mu_{h^\star}'}{n}\big) +\frac{(\mu_f+ 2 \mu_g)\mu_{h^\star}'}{n^2} \Big)
 \geq 0  \,.
\end{align*}
The first term is nonnegative thanks to Assumption~\ref{hyp:main}\eqref{hyp:main-norm}. The second term is nonnegative
as soon as 
\[
\eta \leq \frac 1n \frac{(\mu_f + 2 \mu_g)\mu_{h^\star}'}{\mu_f + 2 \mu_g + \mu_{h^\star}'} \,.
\]

To conclude, we remark that
\begin{equation*}
\frac 1n \frac{(\mu_f + 2 \mu_g)\mu_{h^\star}'}{\mu_f + 2 \mu_g + \mu_{h^\star}'}
\leq \min(\frac 1n \frac{\mu_f + 2 \mu_g}{1 + 2 \mu_g}, \frac 1n \frac{2 \mu_{h^\star}'}{1 + 2 \mu_{h^\star}'} ) \leq \frac 1n \,.
\end{equation*}

This result also implies the same rate for the iterates of Algorithm~\ref{algo:main} because $h^*$ is $\mu_{h^\star}$-strongly convex in the norm $\norm{\cdot}_{\sigma^{-1}}$ if and only if $\overline h^\star$ is $\mu_{h^\star}$-strongly convex in the norm $\norm{\cdot}_{\tilde \sigma^{-1}}$.
\end{proof}

\begin{remark}
It is worth noting that the algorithm does not depend on the strong convexity constants, 
which means that it automatically adapts to local strong convex-concave parameters of the Lagrangian. Moreover as can be seen on Figure~\ref{fig:svm_rcv1} we do observe linear convergence in some cases,
even when Assumption~\ref{hyp:strconv} is not satisfied. Thus we think that
Theorem~\ref{the:strconv} can give an indication of how the algorithm
behaves in favorable cases.
\end{remark}

\begin{remark}
Of particular interest is the relation between the rate proved 
in Theorem~\ref{the:strconv} and the size of the steps. 
Having longer step sizes improves the rate greatly since $\mu_f$, $\mu_g$
and $\mu_{h^\star}$, measured in the weighted norm, are ``proportional'' to the 
step-sizes: as $\mu_g \norm{x}_{\tau^{-1}}^2 = (\alpha \mu_g) \norm{x}_{(\alpha \tau)^{-1}}^2$ for all $\alpha>0$, multiplying the step-sizes by $\alpha > 1$ also multiplies  $\mu_f$, $\mu_g$
and $\mu_{h^\star}$ by $\alpha$, which leads to an improved rate
$1- \frac 1n \frac{(\alpha \mu_f + 2 \alpha \mu_g) \alpha\mu_{h^\star}}{\alpha \mu_f + 2 \alpha \mu_g + \alpha \mu_{h^\star}} = 1 - \alpha  \frac 1n \frac{(\mu_f + 2 \mu_g)\mu_{h^\star}}{\mu_f + 2 \mu_g + \mu_{h^\star}} < 1 -  \frac 1n \frac{(\mu_f + 2 \mu_g)\mu_{h^\star}}{\mu_f + 2 \mu_g + \mu_{h^\star}}$.

 As shown in Section \ref{sec:svm}, in large scale applications
one can expect much more than twice larger steps and so we can expect
a much faster algorithm by using large steps than by using the steps proposed
in \cite{iutzeler2013asynchronous}.
\end{remark}

}

\newpage 

\section{Numerical Experiments}
\label{sec:simu}

For all the experiments, we used one processor of a computer with Intel Xeon CPUs at 2.80GHz.

\subsection{Total Variation + $\ell_1$ Regularized Least Squares Regression}
\label{sec:exp_tv}

For given regularization parameters $\alpha>0$ and $r \in [0,1]$, we would like to solve the following  regression problem with regularization given by the sum of Total Variation (TV) and the $\ell_1$ norm:
\[
\min_{x \in \mathbb{R}^n} \frac{1}{2}\norm{Ax - b}_2^2 + \alpha\big( r \norm{x}_1 + (1-r) \norm{M x}_{2,1}\big).
\]
The problem takes place on a 3D image of the brains of size $40 \times 48 \times 34$. 
The optimization variable $x$ is a real vector with one entry in each voxel, that is $n$ = 65,280. 
Matrix $M$ is the discretized 3D gradient. This is a sparse matrix of size 195,840 $\times$ 65,280
with 2 nonzero elements in each row.
The matrix $A \in \mathbb{R}^{768 \times 65,280}$ and the vector $b \in \mathbb{R}^{768}$ correspond to 768 labeled experiments where each line of $A$ 
gathers brains activity for the corresponding experiment. Parameter $r$ tunes the tradeoff between the two regularization terms.
If $r$ = 1, one gets a Lasso problem for which coordinate descent has been reported to be very efficient~\cite{Friedman_Hastie_Hofling_Tibshirani07}.
For $r<1$, classical (primal) coordinate descent cannot be applied but primal-dual coordinate descent can. 

In this scenario, we set the objective as $f(x) = \frac{1}{2}\norm{Ax - b}_2^2 $, $g(x) = \alpha  r \norm{x}_1 $
and $h(y) = \alpha (1-r) \norm{y}_{2,1}$.
We coded Algorithm~\ref{algo:main} 
in Cython\footnote{The code is available on \url{http://perso.telecom-paristech.fr/~ofercoq/Software.html}}
and duplicated each dual variable two times.
Note that as $h=\alpha (1-r) \norm{\cdot}_{2,1}$ is not separable, we need to compute 12 dual components of $\bar{\bs y}_{k+1}$ for each
primal variable $x^{(i)}_{k+1}$ updated and then use only 6 of them to update $z^{(j)}_{k+1}$ for $j \in J(i_{k+1})$.
\revisiontrois{This procedure is explained in detail in Section~\ref{sec:general_case}.}
We chose $\sigma_j$ such that $\rho(\sum_{j \in J(i)} \sigma_j M_{j,i}^\star M_{j,i})$ is of the same order of magnitude as  
$\beta_i$  and $\tau_i$ equal to 0.95 times its upper bound in Assumption~\ref{hyp:main}. 
We compared Algorithm~\ref{algo:main} against:
\begin{itemize}
\item V\~{u}-Condat's algorithm~\cite{vu2013splitting,condat2013primaldual},
\item Chambolle-Pock's algorithm~\cite{chambolle2011first}, 
\item FISTA~\cite{beck2009fista} with an inexact resolution of the proximal operator of TV 
\revision{and a momentum factor ensuring convergence~\cite{chambolle2014convergence}},
\item L-BFGS~\cite{zhu1997lbfgs} with a smoothing of the nonsmooth functions and continuation.
\end{itemize}
Figure~\ref{fig:tvl1} indicates that our primal coordinate descent is a competitive algorithm
for a wide range of regularization parameters. 

Note that Chambolle-Pock needs to
compute the singular values decomposition of $A$ (which explains the flat shape of the 
performance curve when the algorithm starts). FISTA and V\~{u}-Condat need to 
estimate its largest singular value. If only a low accuracy is required, Algorithm~\ref{algo:main}
may have reached this low accuracy even before these preprocessing steps are completed.

L-BFGS has similar behaviour as Algorithm~\ref{algo:main} except for $\alpha=0.1$, $r=0.9$
where it suffers from the non-smoothness of the objective while  Algorithm~\ref{algo:main}
deals with it directly by the proximal operators.
FISTA is the fastest algorithm for \revision{problems with a
heavy TV regularization}.

\begin{figure}
\includegraphics[width=\linewidth, trim = 0 20 0 10, clip]{
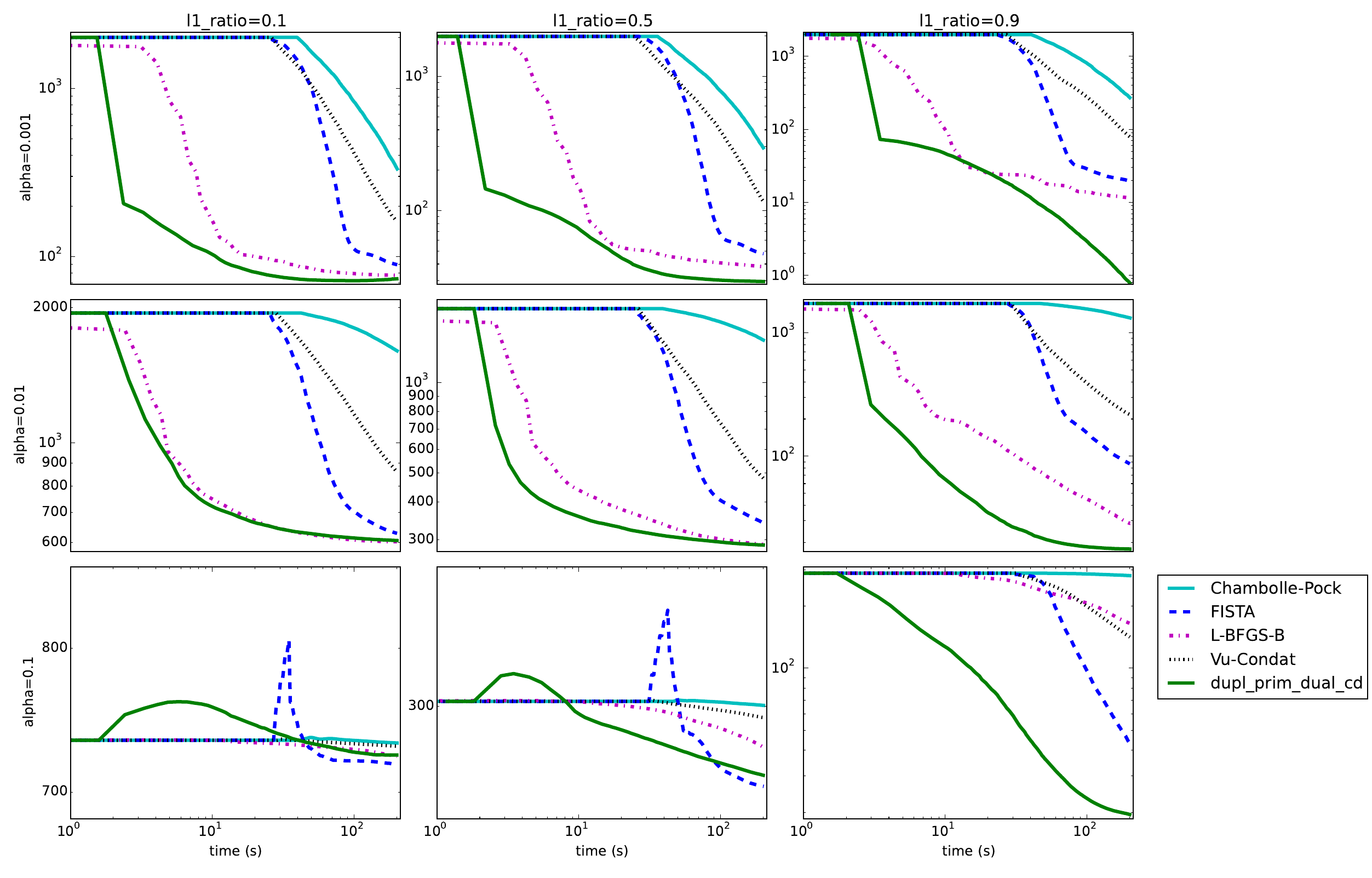}
\caption{Comparison of algorithms for TV+$L_1$-regularized regression at various
regularization parameters. \revision{For each problem, we compute the dual function
at the last iterate (this amounts to solving a Lasso problem). Then we compare
the primal objective curves to this reference value and we plot them in
logarithmic scale. Note for the choices of regularization parameters such that $\alpha (1-r)$ is larger, the problem is more difficult to solve because the total variation regularizer is dominant. This is in fact the most challenging part of the objective because it is non-differentiable and non-separable.}
}
\label{fig:tvl1}
\end{figure}

\subsection{Linear Support Vector Machines}
\label{sec:svm}

We now present a second application for our algorithm. We consider a set of $n$ observations gathered
 into a data matrix $A \in \mathbb{R}^{m \times n}$ and labels $b \in \mathbb{R}^n$ and we intend to 
solve the following Support Vector Machine (SVM) problem:
\[
\min_{w \in \mathbb{R}^m, w_0 \in \mathbb{R} }\sum_{i = 1}^n C_i \max\big(0, 1 - b_i ((A^\top w)_i+w_0) \big) + \frac{\lambda}{2}\norm{w}_2^2.
\]

As is common practice for this problem, we solve instead the Dual Support Vector Machine problem:
\[
\max_{x \in \mathbb{R}^n} -\frac{1}{2 \lambda} \norm{ A D(b)x}^2_2 + e^T x - \sum_{i = 1}^n I_{[0, C_i]}(x_i) - I_{\{0\}}(\langle b, x \rangle)
\]
\revision{Here, we are considering a nonzero bias. Therefore the primal SVM problem is
not strongly convex and the Dual SVM problem has a coupling constraint. Some authors
proposed to fix the bias to 0 in order to make the problem easier to solve but
we show that our method can solve the original SVM problem nearly as fast.}

In the experiments\footnote{Code available on \url{https://github.com/ofercoq/lightning}}, we consider: 
\begin{itemize}
\item the RCV1 dataset~\cite{lewis2004rcv1} where $A$ is a sparse $m \times n$ matrix with $m$ = 20,242, $n$ = 47,236 and 0.157 \%
of nonzero entries and we take $C_i = \frac{1}{n}$ for all $i$ and $\lambda = \frac{1}{4n}$. \revision{For this dataset, $\norm{A}^2 \approx 450 \max_i \norm{Ae_i}^2$, which means that using small step sizes leads to a roughly 
450 times slower algorithm. This situation is not uncommon and is one of the reasons why coordinate descent methods are attractive.}
\item the KDD cup 2009 dataset~\cite{guyon2009analysis}: the data is a mix of  14740 numerical values
and 260 categorical values from Orange Labs. We preprocessed the data by adding a feature for each column containing missing values
and binarizing the categorical values. We obtained a sparse matrix with $m$ = 86,825, $n$~=~50,000 and 1.79 \% of nonzero entries.
We divided the columns by their standard deviation and removed columns with a too small standard deviation.
There are three tasks with this dataset: estimate the appetency, churn and up-selling probability of customers.
As the classes are unbalanced, we compensate this with values of $C_i$ proportional to the class weight
and we chose $\max_i C_i = \lambda = \frac{1}{n}$. We also chose a value of $\sigma_i$ depending on the class.
\end{itemize}
Here $f(x) = \frac{1}{2} \norm{ A D(b)x}^2_2 - e^T x$,
$g(x) = \sum_{i = 1}^n I_{[0, C_i]}(x_i)$, $h(y) = I_{\{0\}}(y)$ ($h$ is the indicator for $\{0\} \subset \mathbb R$, {\em i.e.} $h(y) = 0$ if $y=0$, $h(y) = +\infty$ otherwise) and $M= b^\top$.
We compare the following methods:
\begin{itemize}
\item SCDA~\cite{Stoch-dual-Coord-Ascent}: note that SDCA simply forgets $I_{\{0\}}(y)$
in order to be able to apply the classical coordinate descent method a thus will
not converge to an optimal solution.
\item RCD~\cite{necoara2014random}: at each iteration, the algorithm selects two coordinates
randomly and performs a coordinate descent step according to these two variables.
Updating two variables at a times allows us to satisfy the linear constraint at each iteration.
\revisiontrois{
\item Primal-dual coordinate descent (PD-CD) with small steps using the step size $\tau < \frac 1{\beta(f) / 2 + \sigma\rho\left( K^\star K\right)}$ as in~\cite{iutzeler2013asynchronous}.
\item Algorithm~\ref{algo:main} with $\mathcal J(i) = \{i\} \times J(i)$ for all $i$ (PC-CD).
\item Algorithm~\ref{algo:notduplicated}, i.e. Algorithm~\ref{algo:main} with $\mathcal J(i) = \cup_{1 \leq j \leq n} \{j\} \times J(j)$ for all $i$ in order to maintain a single
Lagrange multiplier (PD-CD without duplication).
}
\end{itemize}
We can see on Figures~\ref{fig:svm_rcv1} and~\ref{fig:svm_orange} the decrease of the SVM duality gap for each algorithm.
SDCA is very efficient in the beginning and converges quickly. However, 
as the method does not take into account the intercept, it does not converge to 
the optimal solution and stagnates after a 
few passes on the data. 
 Algorithm~\ref{algo:main} allows step sizes nearly as long as SDCA's and 
taking into account the coupling constraint represents only marginal additional work.
Hence, the objective value decreases nearly as fast for SDCA in the beginning
without sacrificing the intercept, leading to a smaller objective value in the end.
The RCD method of~\cite{necoara2014random} does work but is not competitive
in terms of rate of convergence. \revisiontrois{Also, as expected, using small steps~\cite{iutzeler2013asynchronous} 
leads to a very slow algorithm in this context. Finally, for this problem,
the additional memory requirement induced by duplication is negligible compared to the size of the problem data, but the
slightly stricter step size condition may explain why PD-CD without duplication is slower.}
We also tried the C implementation of LIBSVM~\cite{chang2011libsvm}
but it needed 175s to solve the (medium-size) RCV1 problem.

\begin{figure}
\centering
\includegraphics[width=0.9\linewidth, trim = 0 15 0 15, clip]{
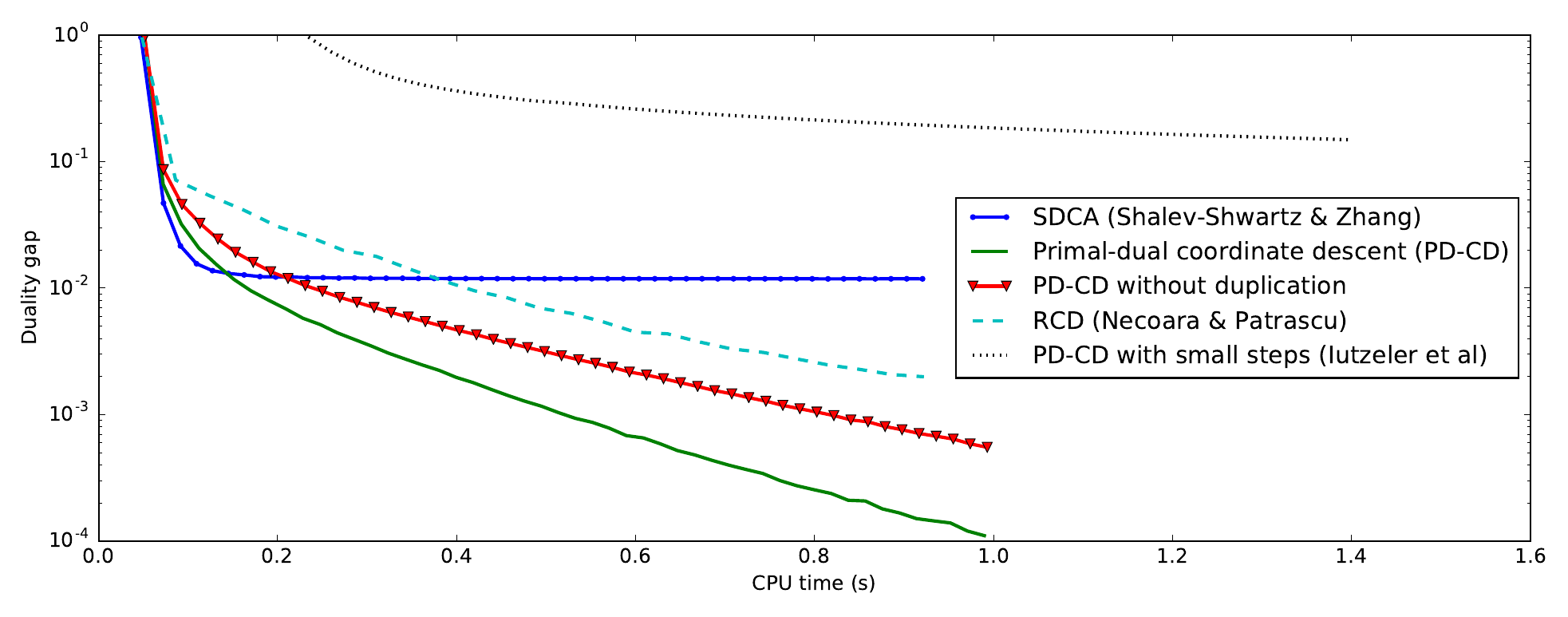}
\caption{Comparison of dual algorithms for the resolution of linear SVM on the RCV1 dataset.
We report the value of the duality gap after a post-processing to recover feasible primal and dual variables.
Primal variables are recovered as suggested in~\cite{Stoch-dual-Coord-Ascent}
and the intercept is recovered by exact minimization of the primal objective given the other primal variables.
When dual iterates are not feasible, we project them onto the dual feasible set before computing the dual objective. 
We stopped each algorithm after 100 passes through the data: note that the cost per iteration of the 5 algorithms is similar \revisiontrois{but that the algorithm of~\cite{iutzeler2013asynchronous} needs first to compute the Lipschitz constant of the gradient.}}
\label{fig:svm_rcv1}
\end{figure}

\begin{figure}
\centering
\includegraphics[width=0.33\linewidth, trim = 20 10 20 20, clip]{
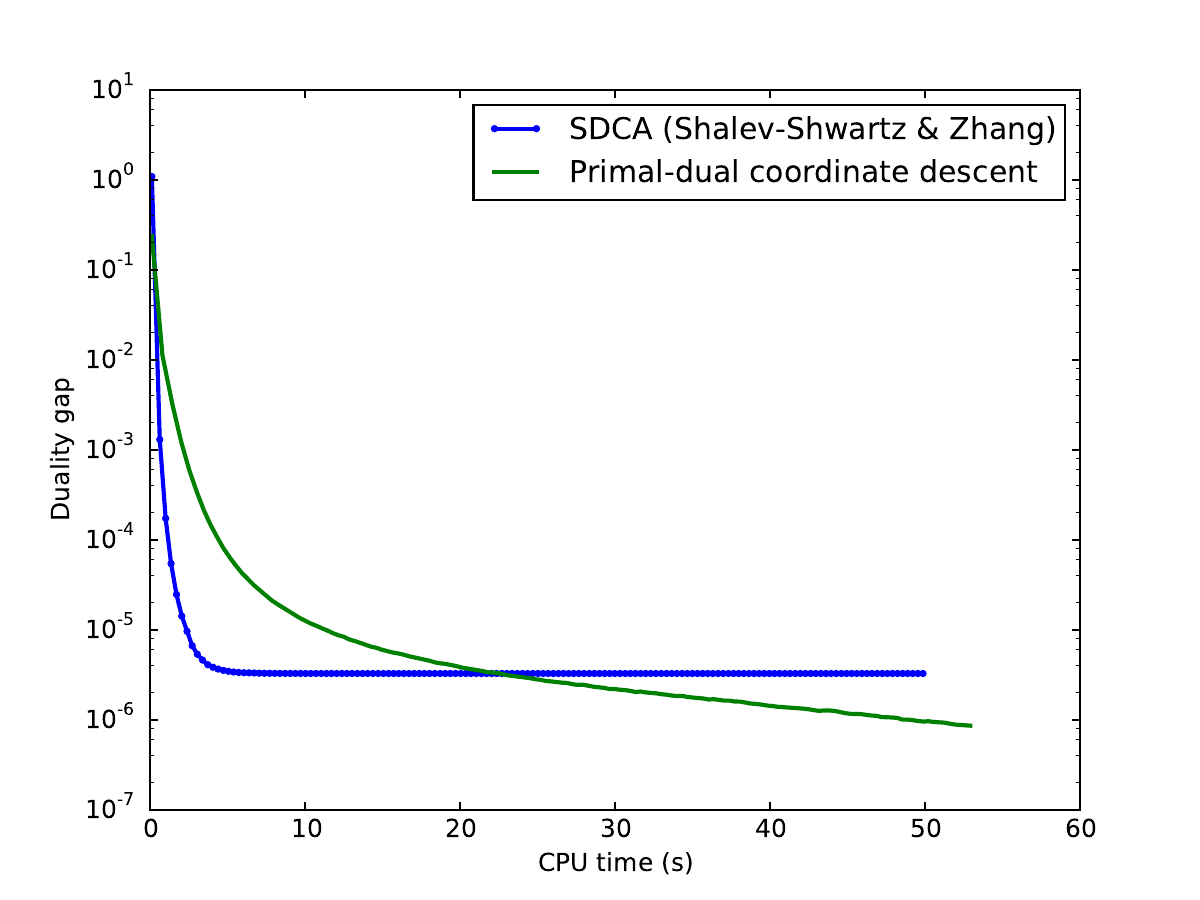}%
\includegraphics[width=0.33\linewidth, trim = 20 10 20 20, clip]{
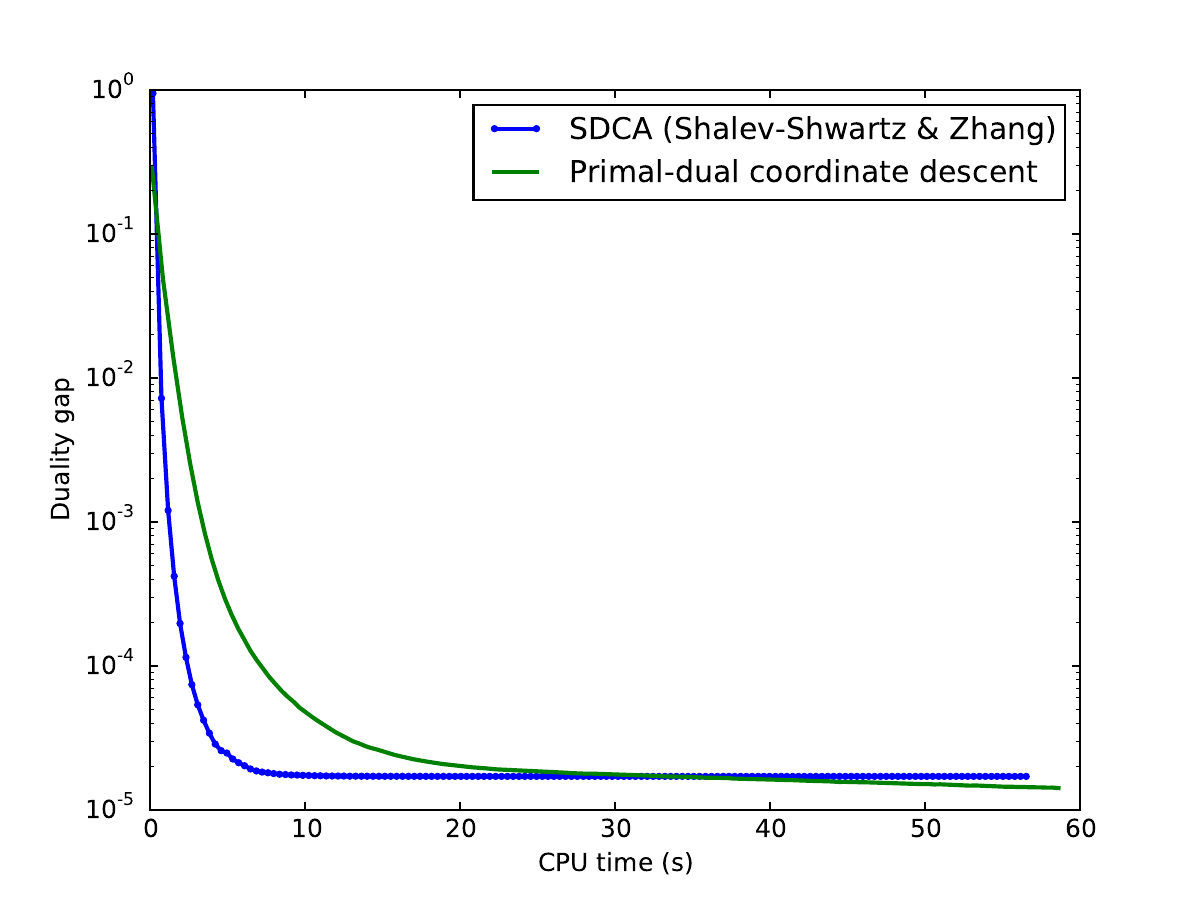}%
\includegraphics[width=0.33\linewidth, trim = 20 10 20 20, clip]{
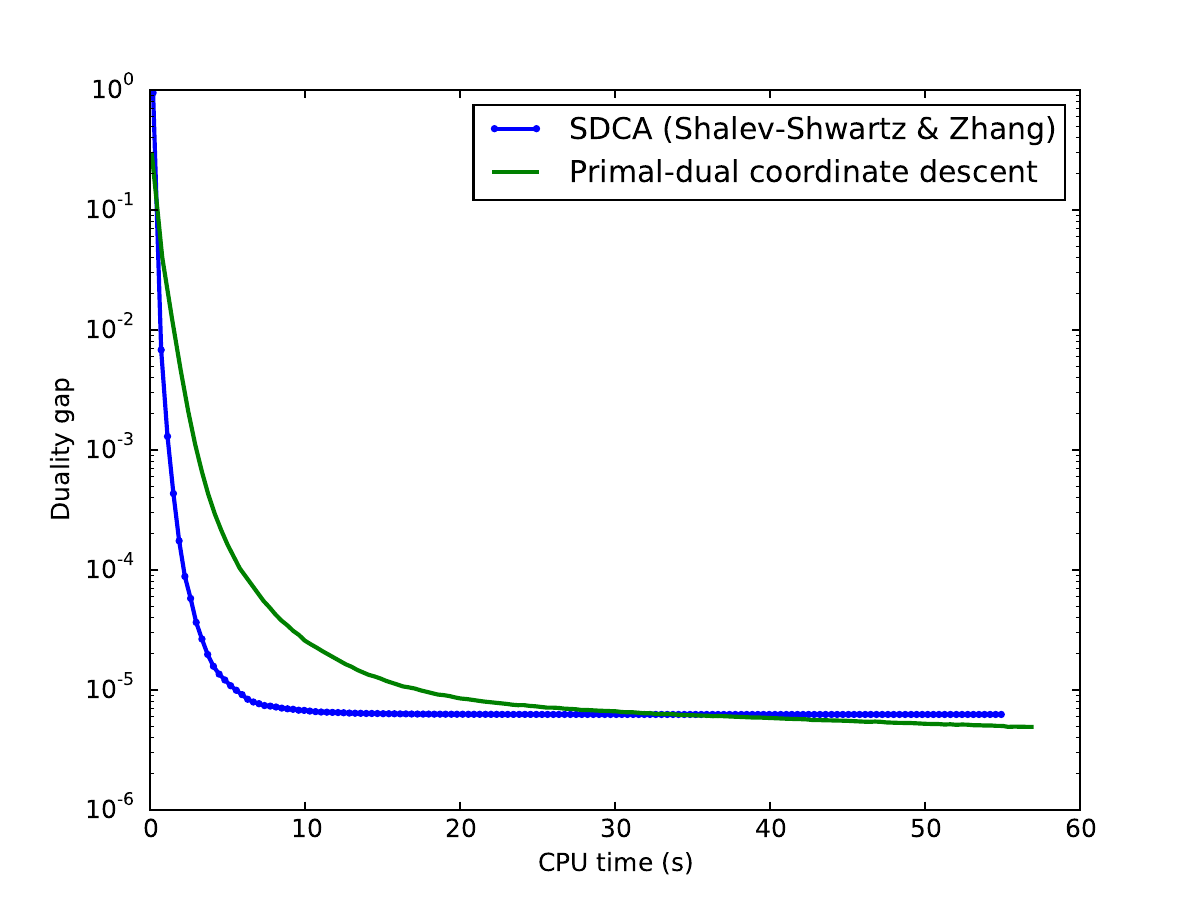}
\caption{Comparison of dual algorithms for the resolution of linear SVM on the KDD cup 2009 dataset 
for the appetency, churn and up-selling tasks (one plot for each).
We did the same post-processing as in Fig.~\ref{fig:svm_rcv1}.
We stopped each algorithm after 300 passes through the data.
We can see here also that dealing with the intercept allows us to 
find more accurate solutions for a similar computational cost as with SDCA.}
\label{fig:svm_orange}
\end{figure}

\revision{
\section{Conclusion}

In this work, we combined features of two seemingly incompatible versions
of coordinate descent: one based on Fej\'er monotonicity~\cite{combettes2014stochastic}, which allows
non-separable non-smooth functions, and one based on the decrease of the
function value~\cite{RT:UCDC}, which allows a large step size.
We proved the convergence of the algorithm and demonstrated its efficiency
on two large scale problems.

Our future work will focus on the limits of Theorem~\ref{the:main}. 
We believe that the restriction to uniform sampling 
probabilities can be removed. Also, by analogy with V\~u-Condat's method, 
one should be able to replace $\beta_i$ by $\beta_i/2$ in the step size condition.
A more prospective research, motivated by \cite{mairal2015incremental},
consists in studying the impact of
the non-smooth functions on the range of step sizes ensuring convergence. 

}

\section*{Acknowledgement}

We are grateful to Elvis Dohmatob for letting us use his benchmarking tool~\cite{dohmatob2014benchmarking}.


\bibliographystyle{siamplain}
\bibliography{literature}

\end{document}